\documentclass[a4paper,12pt, reqno]{amsart}
\usepackage{amsmath,amsfonts,amssymb,amsbsy,amsthm,epsfig,color,mathrsfs}
\usepackage{hyperref}

\newcommand{\ignore}[1]{}

\newcommand{\supp}{\operatorname{supp}}

\newcommand{\GL}{\operatorname{GL}}

\newcommand{\bb}{\mathbb}

\newcommand{\height}{\textrm{\rm H}}
\newcommand{\dist}{\textrm{\rm dist}}
\newcommand{\R}{\bb R}

\newcommand{\Q}{\mathbb Q}

\newcommand{\vol}{\operatorname{vol}}

\newcommand{\adelic}{ad\'{e}lic}

\newtheorem{Theorem}{Theorem}
\newtheorem{Cor}[Theorem]{Corollary}
\newtheorem{Prop}[Theorem]{Proposition}
\newtheorem{Lemma}[Theorem]{Lemma}
\newtheorem*{lemma*}{Lemma}

\newtheorem{example}[Theorem]{Example}
\newtheorem{remark}[Theorem]{Remark}
\newtheorem*{theorem*}{Theorem}

\numberwithin{equation}{section}
\numberwithin{Theorem}{section}

\begin{document}
\title[Diophantine approximation]{Diophantine approximation and automorphic spectrum}
\author{Anish Ghosh, Alexander Gorodnik, and Amos Nevo} 
\address{School of Mathematics, University of East Anglia, Norwich, UK }
\email{a.ghosh@uea.ac.uk}
\address{School of Mathematics, University of Bristol, Bristol UK }
\email{a.gorodnik@bristol.ac.uk}
\address{Department of Mathematics, Technion IIT, Israel}
\email{anevo@tx.technion.ac.il}

\date{\today}
\thanks{The first author acknowledges support of  EPSRC. The second author acknowledges
  support of EPSRC, ERC and RCUK. The third author acknowledges support of ISF}

\begin{abstract}
The present paper establishes quantitative estimates on the rate of diophantine approximation in 
homogeneous varieties of semisimple algebraic groups.
The estimates established generalize and improve previous ones, and are sharp in a number of cases. We show that the rate of diophantine approximation 
is controlled by the spectrum of the automorphic representation, 
and thus subject to the generalised Ramanujan conjectures.
\end{abstract}

\maketitle

{\small
\tableofcontents
 }

\section{Introduction}

Diophantine approximation can be viewed as an attempt to quantify the density of the set $\mathbb{Q}$ of
rational numbers in the reals $\mathbb{R}$ and, more generally, the density of a number field $K$
in its completion $K_v$. In this paper we will be interested in the problem of Diophantine approximation
on more general algebraic varieties. 
Let $\rm X$ be an algebraic variety defined over a number field $K$.
Given a height function $\height: {\rm X}(K)\to \mathbb{R}^+$ and a metric $\dist_v$ on ${\rm X}(K_v)$,
we introduce a function $\omega_v(x,\epsilon)$ which measures the density of $X(K)$ in $X(K_v)$, and thus the 
Diophantine properties of points $x$ in ${\rm X}(K_v)$ with respect to $X(K)$. We define : 
\begin{equation}\label{eq:omega}
\omega_v(x,\epsilon):=\min\{\height(z):\, z\in {\rm X}(K),\, \dist_v(x,z)\le \epsilon \}
\end{equation}
(if no such $z$ exists, we set $\omega_v(x,\epsilon)=\infty$).
This function is a natural generalization
of the uniform irrationality exponent $\hat \omega(\xi)$ of a real number $\xi$
(see, for instance, \cite{BL}).
Note that $\omega_v(x,\epsilon)$ is a non-increasing function which is bounded as $\epsilon\to 0^+$
if and only if $x\in {\rm X}(K)$ and is finite if and only if $x\in \overline{{\rm X}(K)}$.
For $x\in \overline{{\rm X}(K)}\backslash {\rm X}(K)$,
it is natural to consider the  growth rate of $\omega_v(x,\epsilon)$ as $\epsilon\to 0^+$,
which provides a quantitative measure of irrationality of $x$ with respect to $K$.

Our paper is motivated by the work \cite{W} of M. Waldschmidt  who considered
this problem in the case when ${\rm X}$ is an Abelian variety defined over $\Q$ 
equipped with the N\'eron--Tate height.
\ignore{We restate his results using our notation:
assuming a lower bound on the rank of the group ${\rm X}(\mathbb{Q})$,
he proved that for explicit $\theta\in (0,1)$, every $\epsilon\in (0,\epsilon_0)$,
and every $x\in \overline{X(\QQ)}\subset X(\mathbb{R})$,
$$
\omega_\infty(x,\epsilon)\le \exp(\exp(\log(1/\epsilon)^\theta)).
$$
}
In terms of our notation, he proved upper estimates on the function $\omega_\infty(x,\epsilon)$
and conjectured that for every $\delta>0$,
$\epsilon\in (0,\epsilon_0(\delta))$, and $x\in \overline{{\rm X}(K)}\subset {\rm X}(\mathbb{R})$,
$$
\omega_\infty(x,\epsilon)\le \epsilon^{-\frac{2\dim({\rm
        X})}{\hbox{\fontsize{6pt}{6pt}\selectfont rank}({\rm X}(\mathbb{Q}))}-\delta}.
$$
This conjecture is remarkably strong as one can show that
the exponent in this estimate is the best possible.
We also mention that there is a similar conjecture in the case of  algebraic tori
(see \cite[Conjecture~4.21]{W2}).

More generally, let ${\rm X}\subset \mathbb{A}^n$ be a quasi-affine variety defined over a number field $K$. 
We denote by $V_K$ the set of normalised absolute values $|\cdot|_v$ of $K$,  by $K_v$ the $v$-completion  of $K$, by $k_v$ the residue field and by $q_v$ its cardinality. 
We  define the height function on ${\rm X}(K)$:
\begin{equation}\label{eq:height}
\height(x):=\prod_{v\in V_K} \max_{1\le i\le n} (1,|x_i|_v),
\end{equation}
and the metric on ${\rm X}(K_v)$:
\begin{equation}\label{eq:dist}
\|x-y\|_v:=\max_{1\le i\le n} |x_i-y_i|_v.
\end{equation}

In this paper we derive upper estimates on the functions $\omega_v(x,\epsilon)$ for
quasi-affine varieties which are homogeneous spaces of semisimple algebraic groups.
Our upper bounds on the functions $\omega_v(x,\epsilon)$
depend on information about the spectrum of the associated automorphic representations,
and the {\it best possible} upper bounds will be established in several cases. To illustrate the relevance of the Ramanujan-Petersson  conjectures to our analysis let us consider first the case of diophantine approximation on hyperboloids. Further examples, including the case of spheres of dimensions $2$ and $3$ where best possible upper bounds are obtained, will be discussed in \S  \ref{sec:examples}.

\begin{example} \label{eq:first}
{\rm 
Let $Q$ be a non-degenerate quadratic form in three variables defined over a number field
$K\subset\mathbb{R}$, $a\in K$, and 
$$
{\rm X}=\{Q(x)=a\}.
$$
For a finite set of non-Archimedean places of $K$, we denote by $O_S$ the ring of $S$-integers.
We suppose that $Q$ is isotropic over $S$ and ${\rm X}(O_S)\ne \emptyset$.
Then assuming the Ramanujan--Petersson conjecture for $\hbox{PGL}_2$ over $K$,
our main results imply that  (w.r.t. the maximum norm $\|\cdot\|_\infty$ on $\R^3$, the completion at $v=\infty$) 

\begin{itemize}
\item[(i)]  for almost every $x\in {\rm X}(\mathbb{R})$,
$\delta>0$, and $\epsilon\in (0,\epsilon_0(x,\delta))$, there exists $z\in {\rm X}(O_S)$
such that 
\begin{equation*}
\|x-z\|_\infty\le \epsilon\quad\hbox{and}\quad \height(z)\le \epsilon^{-2-\delta},
\end{equation*}
where the exponent $2$ is the best possible (cf. (\ref{eq:lower}) below).
\item[(ii)] 
for every $x\in {\rm X}(\mathbb{R})$ with $\|x\|\le r$,
$\delta>0$, and $\epsilon\in (0,\epsilon_0(r,\delta))$, there exists $z\in X(O_S)$
such that 
\begin{equation*}
\|x-z\|_\infty\le \epsilon\quad\hbox{and}\quad \height(z)\le \epsilon^{-4-\delta}.
\end{equation*}
\end{itemize}
Using the best currently known estimates towards the Ramanujan--Petersson conjecture (see \cite{ks}),
our method gives unconditional solutions to (i) and (ii) with  
$$
\height(z)\le \epsilon^{-\frac{18}{7}-\delta}\quad\hbox{and}\quad \height(z)\le \epsilon^{-\frac{36}{7}-\delta}
$$
respectively.
Moreover, when $K=\mathbb{Q}$, (i) and (ii) give unconditional solutions to the problem of diophantine approximation on the hyperboloid $X(\R)$ (when $Q$ is isotropic over $\R$), with 
$$
\height(z)\le\epsilon^{-\frac{64}{25}-\delta}\quad\hbox{and}\quad \height(z)\le\epsilon^{-\frac{128}{25}-\delta}
$$
respectively, using \cite[Appendix~2]{KimSa}.

We also mention that a positive proportion of all places satisfy the bound predicted by the Ramanujan--Petersson conjecture (see \cite{ram,ks}).
For such $S$, results (i) and (ii) hold unconditionally.

Finally, we note that for  forms unisotropic over $\R$ we will indeed establish the best possible bound unconditionally - see \S  2. 
}
\end{example}

Before we state our results in full generality,
we observe that there is an (obvious) lower bound for the function $\omega_v$.
For a subset $Y$ of ${\rm X}(K_v)$, we set
$$
\omega_v(Y,\epsilon):=\sup_{y\in Y} \omega_v(y,\epsilon).
$$
Assuming that $Y$ is not a subset of ${\rm X}(K)$, one can give a lower estimate on $\omega_v(Y,\epsilon)$
that depends on the set $Y$, and more specifically on 
the Minkowski dimension $d(Y)$ of $Y$ and 
the size of the set of relevant approximating rational points, namely 
the exponent $a(Y)$ of $Y$.

The Minkowski {\it dimension} of a subset $Y$ of ${\rm X}(K_v)$ is defined by
$$
d(Y):=\liminf_{\epsilon\to 0^+} \frac{\log D(Y,\epsilon)}{\log (1/\epsilon)},
$$
where $D(Y,\epsilon)$ denotes the least number of balls of radius $\epsilon$ (w.r.t. the distance $\text{dist}_v$) needed to cover $Y$.
The set of nonsingular points in ${\rm X}(K_v)$ has a structure of analytic manifold over $K_v$.
In particular, it is equipped with a canonical measure class.
We note that if a subset $Y$ of ${\rm X}(K_v)$ has positive measure, then 
$$
d(Y)=r_v\dim({\rm X})
$$
where $r_v=2$ if $K_v\simeq \mathbb{C}$ and $r_v=1$ otherwise.

The {\it exponent} of a subset $Y$ of ${\rm X}(K_v)$ is defined by
\begin{align*}
\mathfrak{a}_v(Y) &:= \inf_{\mathcal{O}\supset  Y} \limsup_{h\to \infty} \frac{\log A_v(\mathcal{O},h)}{\log h},
\end{align*}
where $\mathcal{O}$ runs over open neighborhoods of $Y$ in ${\rm X}(K_v)$, and 
$$
A_v(\mathcal{O},h):=|\{z\in {\rm X}(K):\, \height(z)\le h,\, z\in\mathcal{O}  \}|.
$$
Note that since $Y\nsubseteq {\rm X}(K)$, we have $\omega_v(Y,\epsilon)\to\infty$ as $\epsilon\to 0^+$.
Hence, for a sufficiently small neighbourhood $\mathcal{O}$ of $Y$, every $\delta_1,\delta_2>0$ and $0<\epsilon<\epsilon_0(\mathcal{O},\delta_1,\delta_2)$, we have
$$
\epsilon^{-d(Y)+\delta_1}\le D(Y,\epsilon)\le A_v(\mathcal{O},\omega_v(Y,\epsilon))\le
\omega_v(Y,\epsilon)^{\mathfrak{a}_v(Y)+\delta_2}.
$$
This implies that for every $\delta>0$ and sufficiently small $\epsilon>0$ depending on $\delta$,
\begin{equation*}
\omega_v(Y,\epsilon)\ge \epsilon^{-\frac{d(Y)}{\mathfrak{a}_v(Y)}+\delta}.
\end{equation*}
In particular, when $Y$ has positive measure, we always have the lower bound
\begin{equation}\label{eq:lower}
\omega_v(Y,\epsilon)\ge \epsilon^{-r_v \frac{\dim({\rm X})}{\mathfrak{a}_v(Y)}+\delta}
\end{equation}
for every $\delta>0$ and $\epsilon\in (0,\epsilon_0(Y,\delta))$.

More generally, we consider the problem of diophantine approximation
for points $x=(x_v)_{v\in S}$ with $S\subset V_K$ and $x_v\in {\rm X}(K_v)$.
Let
\begin{align}\label{eq:x_s}
{ X}_S&:=\{(x_v)_{v\in S}:\, x_v\in {\rm X}(K_v);\;\; x_v\in {\rm X}(O_v)\;\hbox{for almost
  all $v$}\},
\end{align}
where $O_v=\{x\in K_v: |x|_v\le 1\}$ is the ring of integers in $K_v$ for non-Archimedean $v$. 
The set ${ X}_S$, equipped with the topology of the restricted direct product,
is a locally compact second countable space.
One of the fundamental questions in arithmetic geometry is to understand the closure $\overline{{\rm X}(K)}$
in ${ X}_S$ where ${ X}(K)$ is embedded in ${ X}_S$ diagonally.
We say that the {\em approximation property} with respect to $S$ holds if $\overline{{\rm X}(K)}={ X}_S$.
Alternatively, denote the ring of $S$-integers of $K$ by 
$$
O_S=\{x\in K:\,\, |x|_v\le 1\hbox{ for non-Archimedean $v\notin S$}\}\footnote{Note that $O_{\{v\}}\neq O_v$, the ring of integers  defined above !} \,.
$$ 

Then the approximation property with respect to $S$ 
can be reformulated as follows: for  
every $x\in { X}_S$ and every $\epsilon=(\epsilon_v)_{v\in S'}$,
where $\epsilon_v \in (0,1)$ and $S'$ is a finite subset of $S$,
there exists $z\in {\rm X}(O_{(V_K\backslash S)\cup S'})$ such that
\begin{align*}
\|x_v-z\|_v\le \epsilon_v\;\;\hbox{for all $v\in S'$}.
\end{align*}
Our aim is to establish a quantitative version of this property.
Given $x$ and $(\epsilon_v)_{v\in S'}$ as above,
we consider
\begin{equation}\label{eq:omega_S}
\omega_S(x,(\epsilon_v)_{v\in S'}):=\min\left\{\height(z):\, z\in {\rm X}(O_{(V_K\backslash S)\cup S'}), 
\|x_v-z\|\le \epsilon_v,\; v\in S'
 \right\}.
\end{equation}
\begin{remark}
{\rm To clarify our notation somewhat, consider the  group variety $G\subset GL_n$. Fixing the finite set $S^\prime\subset V_K$, we want to establish a rate for simultaneous diophantine approximation of all  (or almost all) elements in the group $\prod_{v\in S^\prime} G_v$. The elements in the group $G(K)$ of $K$-rational points which are allowed   in the approximation process are determined by the choice of a  set $S$ containing $S^\prime$, which may be finite or infinite.  One choice is $S=S^\prime$, in which case the approximations property calls for using elements from $O_{V_K}=G(K)$, namely there are no restrictions at all on the $K$-rational matrices allowed. 
Thus for example when $S=\{v\}=S^\prime$, the approximation rate $\omega_S$ defined above is given by the function $\omega_v(x,\epsilon)$ as defined in \eqref{eq:omega}.
Another choice is $S=V_K\setminus{v_0}$, where $v_0$ is a valuation not in $S^\prime$.  In this case the approximation property calls for using only $K$-rational matrices whose elements are in $O_v$ (namely $v$-integral) for every $v\in V_K$, with the exception of $v \in S^\prime$ and $v=v_0$. 
We will of course assume that $G$ is isotropic over $K_{v_0}$ in this case.
We also admit  any other intermediate choice of $S$ containing $S^\prime$, namely we allow imposing {\it arbitrary} integrality conditions on the set of approximating $K$-rational matrices. The integrality  conditions are that the matrices should be $v$-integral for $v\in S\setminus S^\prime$, and we assume that $G$ is isotropic over  $V_K\setminus S$. 
}
\end{remark}

Given a variety $X$, 
we define the {\it exponent} of a subset $Y$ of  ${ X}_S$ as
\begin{align}\label{eq:ass_omega}
\mathfrak{a}_S(Y) &:= \inf_{\mathcal{O}\supset Y}\limsup_{h\to \infty} \frac{\log A_S(\mathcal{O},h)}{\log h},
\end{align}
where $\mathcal{O}$ runs over open neighborhoods of $Y$ in ${ X}_S$, and 
$$
A_S(\mathcal{O},h):=|\{z\in {\rm X}(K):\, \height(z)\le h,\;\; z\in\mathcal{O} \}|.
$$
More appropriately, the notation should be $\mathfrak{a}_S(Y,X)$ but we will suppress the dependence on $X$ in the notation. 
We also define the {\it exponent} $\mathfrak{a}_S({\rm X})$ 
of the variety ${\rm X}$ as the supremum of $\mathfrak{a}_S(Y)$ as $Y$ runs over bounded subsets of $X_S$.  Since our variety $X$ will be fixed and we will not consider subvarieties of it, this notation should cause no conflict. 

As in (\ref{eq:lower}), one can show that given 
$Y\subset { X}_S$ and finite $S'\subset S$ such that the projection
of $Y$ to ${ X}_{S'}$ has positive measure, we have
\begin{equation}\label{eq:lower2}
\sup_{y\in Y} \omega_S(y,(\epsilon_v)_{v\in S'})\ge \prod_{v\in S'} \epsilon_v^{-r_v \frac{\dim({\rm X})}{\mathfrak{a}_S(Y)}+\delta}
\end{equation}
for every $\delta>0$ and $\epsilon_v\in (0,\epsilon_0(Y,S',\delta))$.
It particular, it follows that we have the following universal lower bound
\begin{equation}\label{eq:lower3}
\sup_{y\in Y} \omega_S(y,(\epsilon_v)_{v\in S'})\ge \prod_{v\in S'} \epsilon_v^{-r_v \frac{\dim({\rm
      X})}{\mathfrak{a}_S({\rm X})}+\delta},
\end{equation}
provided that the projection of $Y$ to $X_{S'}$ has positive measure.

We shall derive upper estimates on the function $\omega_S$ of comparable quality
in the case when ${\rm X}$ is a homogeneous quasi-affine variety of a semisimple group.
We shall start by considering the case of the group variety itself.

Let ${\rm G}$ be a connected almost simple algebraic $K$-group.
Our main results depend on properties of unitary representations $\pi_v$
of the groups ${ G}_{v}:={\rm G}(K_v)$, which we now introduce
\ignore{We refer to Section \ref{sec:semisimple_and_spherical} below for a detailed
discussion of the structure of the group $G_S$. Here we note that $G_S$ is
a locally compact second countable group equipped with a Haar measure
$m_S$.}
We fix a suitable maximal compact subgroup $U_v$ of $G_v$,
whose choice is discussed in Section \ref{sec:semisimple_and_spherical}.
The group ${\rm G}(K)$ embeds in the restricted direct product group ${G}_{V_K}$ diagonally as a discrete subgroup 
and $\vol({ G}_{V_K}/{\rm G}(K))<\infty$ (see \cite[\S5.3]{PlaRa}).
We consider the Hilbert space $L^2({ G}_{V_K}/{\rm G}(K))$
consisting of square-integrable functions on ${ G}_{V_K}/{\rm G}(K)$.
A unitary continuous character $\chi$ of ${ G}_{V_K}$ is called automorphic if 
$\chi({\rm G}(K))=1$. Then $\chi$ can be considered as an element of $L^2({ G}_{V_K}/{\rm G}(K))$.
We denote by $L_{00}^2({ G}_{V_K}/{\rm G}(K))$ the subspace of $L^2({ G}_{V_K}/{\rm G}(K))$ orthogonal to all automorphic
characters. The translation action of the group ${ G}_{v}$ on ${ G}_{V_K}/{\rm G}(K)$ defines
the unitary representation $\pi_v$ of ${ G}_{v}$ on $L_{00}^2({ G}_{V_K}/{\rm G}(K))$.
We define the spherical integrability exponent of $\pi_v$ w.r.t. $U_v$ as follows
\begin{equation}\label{eq:q_v}
\mathfrak{q}_v({\rm G}):=\inf\left\{ q>0:\, \begin{tabular}{l}
$\forall$\hbox{ $U_v$-inv. $w\in L_{00}^2({ G}_{V_K}/{\rm G}(K))$}\\
\hbox{$\left< \pi_v(g)w,w\right>\in L^q(G_v)$}
\end{tabular}
\right\}.
\end{equation}
If $\mathfrak{q}_v({\rm G})=2$, then we say that the representation $\pi_v$ is {\em tempered} 
\footnote{This is equivalent to the standard notion of a tempered representation
defined in terms of weak containment,  by \cite[Theorem~1]{chh}.}.

It is one of the fundamental results in the theory of automorphic representations
that the integrability exponent $\mathfrak{q}_v({\rm G})$ is finite.
(see \cite[Theorem~3.1]{Clozel}).
Moreover, $\mathfrak{q}_v({\rm G})$ is uniformly bounded over $v\in V_K$ (see \cite{GMO}). 
The precise value of $\mathfrak{q}_v({\rm G})$ is 
related to generalised Ramanujan conjecture and Langlands functoriality conjectures
(see \cite{Sarnak}). For instance,  the generalised Ramanujan conjecture for $\hbox{SL}_2$
is equivalent to $\mathfrak{q}_v(\hbox{SL}_2)=2$ for all $v\in V_K$, and the best 
currently known estimate established in \cite{ks,KimSa} gives $\mathfrak{q}_v(\hbox{SL}_2)\le \frac{18}{7}$ for
general number fields $K$ and $\mathfrak{q}_v(\hbox{SL}_2)\le \frac{64}{25}$ for $K=\mathbb{Q}$.

We define the {\it exponent} of a subset $S$ of $V_K$ by
\begin{equation}\label{eq:sigma_S}
\sigma_S:=\limsup_{N\to \infty} \frac{1}{\log N} |\{v\in S:\, q_v\le N\}|,
\end{equation}
where $q_v$ denotes the cardinality of the residue field for non-Archimedean $v$.
Let
\begin{equation}\label{eq:q_s}
\mathfrak{q}_S({\rm G})=(1+\sigma_S)\sup_{v\in S} \mathfrak{q}_v({\rm G}).
\end{equation} 
This parameter will appear below as a bound on the integrability exponent of the automorphic representation restricted to the group $G_S$

We now turn to state our results, and note that the approximation property for algebraic varieties we defined above, namely the density of the closure of $X(K)$ in $X_S$ has been been  studied extensively in the setting of algebraic groups (see \cite[Ch.~7]{PlaRa}). 
It is known that when $\rm G$ is connected simply connected
almost simple algebraic group defined over $K$, then the approximation
property holds with respect to $S$ provided that $\rm G$ is isotropic over $V_K\backslash S$.
Our first result can be viewed as  a quantitative version of this fact.

It is convenient to  set $I_v=\{q_v^{-n}\}_{n\ge 1}$ for non-Archimedean $v\in V_K$
and $I_v=(0,1)$ for Archimedean $v\in V_K$.

\begin{Theorem}\label{th:g1}
Let $G$ be a connected simply connected almost simple algebraic $K$-group and
$S$ a (possibly infinite) subset of $V_K$ such that $G$ is isotropic over $V_K\backslash S$. Then
\begin{enumerate}
\item[(i)] There exists a subset $Y$ of full measure in ${ G}_S$ such that
for every $\delta>0$, finite $S'\subset S$,  $x\in Y$, and $(\epsilon_v)_{v\in S'}$ with 
$\epsilon_v\in I_v\cap (0,\epsilon_0(x,S',\delta))$, we have
$$
\omega_S(x,(\epsilon_v)_{v\in S'})\le \left(
\prod_{v\in S'} \epsilon_v^{-r_v \frac{\dim({\rm G})}{\mathfrak{a}_S({\rm
      G})}-\delta}\right)^{\mathfrak{q}_{V_K\backslash S}({\rm G})/2}.
$$
In particular, when the representations $\pi_v$, $v\in V_K\backslash S$, are tempered
and $\sigma_S=0$, then the above exponent is the 
{\rm  best possible} (cf. (\ref{eq:lower2}))!
\item[(ii)]
For every $\delta>0$, bounded $\Omega\subset { G}_S$, and
$(\epsilon_v)_{v\in S'}$ with finite $S'\subset S$ and $\epsilon_v\in I_v\cap (0,\epsilon_v^0(\Omega,\delta))$,
where $\epsilon_v^0(\Omega,\delta)\in (0,1]$ and $\epsilon_v^0(\Omega,\delta)=1$ for almost all $v$,
we have
$$
\omega_S(x,(\epsilon_v)_{v\in S'})\le \left(
\prod_{v\in S'} \epsilon_v^{-r_v \frac{\dim({\rm G})}{\mathfrak{a}_S({\rm G})}-\delta}\right)^{\mathfrak{q}_{V_K\backslash S}({\rm G})},
$$
uniformly over  $x\in \Omega$ and finite $S'\subset S$.
\end{enumerate}
\end{Theorem}

\begin{remark}
{\rm 
Comparing the estimate in Theorem \ref{th:g1}(i) and (\ref{eq:lower2}), we conclude that
the integrability exponents always satisfy $\mathfrak{q}_v({\rm G})\ge 2$. While this fact
was previously known (see, for instance, \cite{BLS}),
it is curious that it follows from diophantine approximation considerations as well.
This point is further addressed in Corollary \ref{cor:q_v} below.
}
\end{remark}

We also prove a version of Theorem \ref{th:g1}(i) which is uniform over finite $S'\subset S$.

\begin{Theorem}\label{th:g2}
With the notation as in Theorem \ref{th:g1}, 
there exists a subset $Y$ of full measure in ${G}_S$ such that
for every $\delta>0$, $x\in Y$, and $(\epsilon_v)_{v\in S'}$ with finite $S'\subset S$ and
$\epsilon_v\in I_v\cap (0,\epsilon_v^0(x,\delta))$, where
$\epsilon_v^0(x,\delta)\in (0,1]$ and $\epsilon_v^0(x,\delta)=1$ for almost all $v$,
we have
$$
\omega_S(x,(\epsilon_v)_{v\in S'})\le \left(
\prod_{v\in S'} \epsilon_v^{-\frac{r_v\dim({\rm G})+\sigma_S}{\mathfrak{a}_S({\rm G})}-\delta}\right)^{\mathfrak{q}_{V_K\backslash S}({\rm G})/2}.
$$
\end{Theorem}

More generally, let ${\rm X}\subset\mathbb{A}^n$ be a quasi-affine algebraic variety defined over $K$ equipped with 
a transitive action of a  connected almost simple algebraic $K$-group ${\rm G}\subset\hbox{GL}_n$.
Given a subset $S$ of $V_K$, we consider the problem of Diophantine approximation 
in $X_S$ by the rational points in ${\rm X}(K)$, or equivalently 
the problem of Diophantine approximation in $X_{S'}$, where $S'$ is a finite subset of $S$,
by points in  ${\rm X}(O_{(V_K\backslash S)\cup S'})$.
The closure $\overline{{\rm X}(O_{(V_K\backslash S)\cup S'})}$ in ${\rm X}_{S'}$ 
can be described explicitly (see Lemma \ref{l:closure} below).
In particular, it follows that $\overline{{\rm X}(O_{(V_K\backslash S)\cup S'})}$ is open ${\rm X}_{S'}$
provided that ${\rm G}$ is isotropic over $V_K\backslash S$. 
Our next result gives a quantitative
version of the density (in the closure) where the estimates are sharp in many cases (see Example \ref{eq:first} and
Section \ref{sec:examples}).

\begin{Theorem}\label{th:g3}
Assume that ${\rm G}$ is isotropic over $V_K\backslash S$. Then
\begin{enumerate}
\item[(i)] For every finite $S'\subset S$, there exists a subset $Y$ of full measure 
in  $\overline{{\rm X}(O_{(V_K\backslash S)\cup S'})}\subset { X}_{S'}$
such that for every $\delta>0$, $x\in Y$, and $(\epsilon_v)_{v\in S'}$ with 
$\epsilon_v\in I_v\cap (0,\epsilon_0(x,S',\delta))$, we have
$$
\omega_S(x,(\epsilon_v)_{v\in S'})\le \left(
\prod_{v\in S'} \epsilon_v^{-r_v \frac{\dim({\rm X})}{\mathfrak{a}_S({\rm G})}-\delta}\right)^{\mathfrak{q}_{V_K\backslash S}({\rm G})/2}.
$$
\item[(ii)]
For every finite $S'\subset S$, $\delta>0$, bounded $Y\subset \overline{{\rm X}(O_{(V_K\backslash S)\cup S'})}$, and
$(\epsilon_v)_{v\in S'}$ with $\epsilon_v\in I_v\cap(0,\epsilon_0(Y,S',\delta))$, we have
$$
\omega_S(x,(\epsilon_v)_{v\in S'})\le \left(
\prod_{v\in S'} \epsilon_v^{-r_v \frac{\dim({\rm X})}{\mathfrak{a}_S({\rm G})}-\delta}\right)^{\mathfrak{q}_{V_K\backslash S}({\rm G})},
$$
uniformly over  $x\in Y$.
\end{enumerate}
\end{Theorem}

\begin{remark}{\rm 
We note that the estimates in Theorem \ref{th:g3} are stated in terms of 
$\mathfrak{a}_S({\rm G})$, but not in terms of $\mathfrak{a}_S({\rm X})$.
While the latter quantity is difficult to compute in general, 
in many cases we have $\mathfrak{a}_S({\rm G})\ge \mathfrak{a}_S({\rm X})$.
For instance, this is so  when the rational points in ${\rm G}$ do not concentrate
on a proper subgroup of ${\rm G}$, namely when $\mathfrak{a}_S({\rm G})>\mathfrak{a}_S(\hbox{Stab}_{\rm
  G}(x^0))$ for some $x^0\in {\rm X}(O_{(V_K\backslash S)\cup S'})$.
If $\mathfrak{a}_S({\rm G})\ge \mathfrak{a}_S({\rm X})$ and 
$\mathfrak{q}_{V_K\backslash S}({\rm G})=2$, then Theorem \ref{th:g3}(i) gives the best
possible estimate (cf. \eqref{eq:lower3}).
}
\end{remark}

Comparing Theorem \ref{th:g3}(i) with (\ref{eq:lower3}), we deduce the lower estimates on the
integrability exponents $\mathfrak{q}_v({\rm G})$.

\begin{Cor}\label{cor:q_v}
Let $G<{\rm GL}_n$ be a connected almost simple algebraic $K$-group
and ${\rm X}\subset\mathbb{A}^n$ a quasi-affine
variety on which ${\rm G}$ acts transitively. 
Assume that ${\rm G}$ is isotropic over $v\in V_K$ and that
${\rm X}(O_{V_K\backslash \{v\}})$ is not empty. Then
$$
\mathfrak{q}_v({\rm G})\ge   \frac{2 \mathfrak{a}_{V_K\backslash \{v\}}({\rm G})}{\mathfrak{a}_{V_K\backslash \{v\}}({\rm X})}.
$$
\end{Cor}

For example, Corollary \ref{cor:q_v} implies that $\mathfrak{q}_v(\hbox{SL}_n)\ge 2(n-1)$, which
is known to be sharp (see Section \ref{sec:examples}).

\section{Examples}\label{sec:examples}

\subsection{Diophantine approximation on spheres}
Let ${\rm S}^d$ be the unit sphere of dimension $d$ centered at origin, $d\ge 2$, which we view as the level set of the standard quadratic form given by the sum of squares.
We fix a prime $p=1\mod 4$. Then ${\rm S}^d(\mathbb{Z}[1/p])$ is dense in ${\rm S}^d(\mathbb{R})$,
and here we derive a quantitative density estimates. We treat the cases $d=2$, $d=3$, and $d\ge 4$
separately.

For $d=2$, Theorem \ref{th:g3} implies that
\begin{itemize}
\item For almost every $x\in {\rm S}^2(\mathbb{R})$, $\delta>0$, and 
$\epsilon\in (0,\epsilon_0(x,\delta))$, there exists $z\in {\rm S}^2(\mathbb{Z}[1/p])$
such that
$$
\|x-z\|_\infty\le \epsilon\quad\hbox{and}\quad \height(z)\le \epsilon^{-2-\delta}.
$$
We note that $\dim ({\rm S}^2)=2$ and $\mathfrak{a}_{V_\mathbb{Q}\backslash \{p\}}({\rm S}^2)=1$, so that
this exponent is the best possible (see (\ref{eq:lower3})).

\item For every $x\in {\rm S}^2(\mathbb{R})$, $\delta>0$, and 
$\epsilon\in (0,\epsilon_0(\delta))$, there exists $z\in {\rm S}^2(\mathbb{Z}[1/p])$
such that
$$
\|x-z\|_\infty\le \epsilon\quad\hbox{and}\quad \height(z)\le \epsilon^{-4-\delta}.
$$
\end{itemize} 
To deduce these estimates, we consider the group 
${\rm G}\simeq {\rm D}^\times/{\rm Z}^\times$,
where ${\rm D}$ denotes Hamilton's quaternion algebra and ${\rm Z}$ the centre of ${\rm D}$.
This group naturally acts on the variety of pure quaternions of norm one, which can be identified
with the sphere ${\rm S}^2$. Hence, we are in position to apply Theorem \ref{th:g3}.
Since $p=1\mod 4$, the Quaternion algebra split over $p$ and ramifies at $\infty$.
In this case, we have $\mathfrak{q}_{V_\mathbb{Q}\backslash \{p\}}({\rm G})=2$, which is a consequence of the results of
Deligne combined with Jacquet--Langlands correspondence (see \cite[Appendix]{Lub}),
$\dim ({\rm S}^2)=2$, and $\mathfrak{a}_{V_\mathbb{Q}\backslash \{p\}}({\rm G})=1$.

For $d=3$, Theorem \ref{th:g3} implies that
\begin{itemize}
\item For almost every $x\in {\rm S}^3(\mathbb{R})$, $\delta>0$, and 
$\epsilon\in (0,\epsilon_0(x,\delta))$, there exists $z\in {\rm S}^3(\mathbb{Z}[1/p])$
such that
$$
\|x-z\|_\infty\le \epsilon\quad\hbox{and}\quad \height(z)\le \epsilon^{-\frac{3}{2}-\delta}.
$$
Since $\dim ({\rm S}^3)=3$ and $\mathfrak{a}_{V_\mathbb{Q}\backslash \{p\}}({\rm S}^3)=2$, 
this exponent is the best possible (see (\ref{eq:lower3})).
\item For every $x\in {\rm S}^3(\mathbb{R})$, $\delta>0$, and 
$\epsilon\in (0,\epsilon_0(\delta))$, there exists $z\in {\rm S}^3(\mathbb{Z}[1/p])$
such that
$$
\|x-z\|_\infty\le \epsilon\quad\hbox{and}\quad \height(z)\le \epsilon^{-3-\delta}.
$$
\end{itemize} 
To deduce these estimates, we now consider the group 
${\rm G}$ of norm one elements of  Hamilton's quaternion algebra ${\rm D}$,
which can be identified with the variety ${\rm S}^3$.
We have $\mathfrak{q}_p({\rm G})=2$, which is a again consequence of the results of
Deligne combined with Jacquet--Langlands correspondence (see \cite[Appendix]{Lub}),
$\dim ( {\rm G})=3$ , and $\mathfrak{a}_{V_\mathbb{Q}\backslash \{p\}}({\rm G})=2$.
Hence, the claimed estimates follow from Theorem \ref{th:g3}.

Now let $d\ge 4$. We consider the natural action of the group ${\rm G}={\rm SO}_{d+1}$
on ${\rm S}^d$. We note that since $p=1\mod 4$, the group ${\rm G}$ splits over $\mathbb{Q}_p$.
By \cite{li,oh}, we have $\mathfrak{q}_p({\rm G})\le d$ for even $d$ and 
$\mathfrak{q}_{V_\mathbb{Q}\backslash \{p\}}({\rm G})\le d+1$ for odd $d$.
By \cite{GN2}, the parameter $\mathfrak{a}_p({\rm G})$ can be estimated in terms of volumes
of the height balls, which in turn can be estimated in terms of the root system data of ${\rm G}$.
This gives the estimates $\mathfrak{a}_{V_\mathbb{Q}\backslash \{p\}}({\rm G})= d^2/4$ for even $d$ and
$\mathfrak{a}_p({\rm G})= (d+1)(d+3)/4$ for odd $d$. Therefore, 
Theorem \ref{th:g3} implies that for even $d\ge 4$, 
\begin{itemize}
\item For almost every $x\in {\rm S}^d(\mathbb{R})$, $\delta>0$, and 
$\epsilon\in (0,\epsilon_0(x,\delta))$, there exists $z\in {\rm S}^d(\mathbb{Z}[1/p])$
such that
$$
\|x-z\|_\infty\le \epsilon\quad\hbox{and}\quad \height(z)\le \epsilon^{-2-\delta}.
$$
The exponent $2$ should be compared with the lower estimate
$\frac{d}{d-1}$ given by (\ref{eq:lower3}). At present we don't know whether this exponent can be improved.
\item For every $x\in {\rm S}^d(\mathbb{R})$, $\delta>0$, and 
$\epsilon\in (0,\epsilon_0(\delta))$, there exists $z\in {\rm S}^d(\mathbb{Z}[1/p])$
such that
$$
\|x-z\|_\infty\le \epsilon\quad\hbox{and}\quad \height(z)\le \epsilon^{-4-\delta}.
$$
\end{itemize} 
Similarly, for odd $d\ge 4$, 
\begin{itemize}
\item For almost every $x\in {\rm S}^d(\mathbb{R})$, $\delta>0$, and 
$\epsilon\in (0,\epsilon_0(x,\delta))$, there exists $z\in {\rm S}^d(\mathbb{Z}[1/p])$
such that
$$
\|x-z\|_\infty\le \epsilon\quad\hbox{and}\quad \height(z)\le \epsilon^{-\frac{2d}{d+3}-\delta}.
$$
\item For every $x\in {\rm S}^d(\mathbb{R})$, $\delta>0$, and 
$\epsilon\in (0,\epsilon_0(\delta))$, there exists $z\in {\rm S}^d(\mathbb{Z}[1/p])$
such that
$$
\|x-z\|_\infty\le \epsilon\quad\hbox{and}\quad \height(z)\le \epsilon^{-\frac{4d}{d+3}-\delta}.
$$
\end{itemize} 
It is interesting to compare our results with the results on Diophantine approximation
on the spheres obtained in \cite{Sch} by elementary methods that use rational parametrisations
of spheres. It is shown in \cite{Sch} that 
for every $x\in {\rm S}^d(\mathbb{R})$ and $\epsilon\in (0,1)$,
there exists $z\in {\rm S}^d(\mathbb{Q})$ such that
$$
\|x-z\|_\infty\le \epsilon\quad\hbox{and}\quad \height(z)\le \hbox{const}\; \epsilon^{-2\lceil\log_2 (d+1)\rceil}.
$$
While this result deals with the set of all $\mathbb{Q}$-points on ${\rm S}^d$
rather than just $\mathbb{Z}[1/p]$-points, the exponent obtained is significantly
weaker than ours.

\subsection{Diophantine approximation in the orthogonal group}
Let ${\rm SO}_{d+1}$, $d\ge 4$,  be the orthogonal group and $p$ be a prime such that $p=1\mod 4$.
Then ${\rm SO}_{d+1}(\mathbb{Z}[1/p])$ is dense in ${\rm SO}_{d+1}(\mathbb{R})$. 
We have $\dim (\hbox{SO}_{d+1})=d(d+1)/2$, and $\mathfrak{a}_{V_\mathbb{Q}\backslash \{p\}}({\rm SO}_{d+1})$,
$\mathfrak{q}_p({\rm SO}_{d+1})$ are given above. Therefore, Theorem \ref{th:g3} gives
\begin{itemize}
\item For almost every $x\in {\rm SO}_{d+1}(\mathbb{R})$, $\delta>0$, and 
$\epsilon\in (0,\epsilon_0(x,\delta))$, there exists $z\in {\rm SO}_{d+1}(\mathbb{Z}[1/p])$
such that
\begin{align*}
\|x-z\|_\infty\le \epsilon\quad\hbox{and}\quad \height(z)\le \epsilon^{-(d+1)-\delta}, \quad\hbox{when
  $d$ is
  even,}\\
\|x-z\|_\infty\le \epsilon\quad\hbox{and}\quad \height(z)\le \epsilon^{-\frac{d(d+1)}{d+3}-\delta}, \quad\hbox{when $d$
  is odd.}
\end{align*}
This should be compared with lower estimates on the exponents given by (\ref{eq:lower3}),
which is $2(d+1)/d$ for even $d$, and $2d/(d+3)$ for odd $d$.
\item For every $x\in {\rm SO}_{d+1}(\mathbb{R})$, $\delta>0$, and 
$\epsilon\in (0,\epsilon_0(\delta))$, there exists $z\in {\rm SO}_{d+1}(\mathbb{Z}[1/p])$
such that
\begin{align*}
\|x-z\|_\infty\le \epsilon\quad\hbox{and}\quad \height(z)\le \epsilon^{-2(d+1)-\delta}, \quad\hbox{when $d$
  is even,}\\
\|x-z\|_\infty\le \epsilon\quad\hbox{and}\quad \height(z)\le \epsilon^{-\frac{2d(d+1)}{d+3}-\delta}, \quad\hbox{when $d$
  is odd.}
\end{align*}
\end{itemize}

\subsection{Diophantine approximation by the Hilbert modular group}
Let $K$ be a totally real number field, $O$ its ring of integers, and $S$
a proper subset of $V_K^\infty$. The Hilbert modular group $\hbox{SL}_2(O)$ is a dense subgroup
of $(\hbox{SL}_2)_S=\prod_{v\in S} \hbox{SL}_2(K_v)$. We have $\dim ({\rm SL}_2)=3$
and $\mathfrak{a}_{S\cup V_K^f}({\rm SL}_2)=2$, and if we assume the Ramanujan conjecture for $\hbox{SL}_2$ over $K$,
then $\mathfrak{q}_v({\rm SL}_2)=2$. Hence,
Theorem \ref{th:g1} implies the following quantitative density results
(under the Ramanujan conjecture):
\begin{itemize}
\item For almost every $x_v\in \hbox{SL}_2(K_v)$ with $v\in S$, $\delta>0$, and 
$\epsilon_v\in (0,\epsilon_0(x,\delta))$, there exists $z\in \hbox{SL}_2(O)$
such that
$$
\|x_v-z\|_v\le \epsilon_v\hbox{ for $v\in S$}\quad\hbox{and}\quad \height(z)\le \prod_{v\in S} \epsilon_v^{-\frac{3}{2}-\delta}.
$$
The exponent $\frac{3}{2}$ is the best possible by  (\ref{eq:lower3}).
\item For every $x_v\in \hbox{SL}_2(K_v)$ with $v\in S$ such that $\|x_v\|_v\le r$, $\delta>0$, and 
$\epsilon_v\in (0,\epsilon_0(r,\delta))$, there exists $z\in \hbox{SL}_2(O)$
such that
$$
\|x_v-z\|_v\le \epsilon_v\hbox{ for $v\in S$}\quad\hbox{and}\quad \height(z)\le \prod_{v\in S} \epsilon_v^{-3-\delta}.
$$
\end{itemize}
Using the best currently known estimates towards the Ramanujan--Petersson conjecture (see \cite{ks}),
we have $\mathfrak{q}_v({\rm SL}_2)\le \frac{18}{7}$, and 
Theorem \ref{th:g1}  gives unconditional solutions to the above inequalities
$\|x_v-z\|_v\le \epsilon_v$, $v\in S$, with 
$$
\height(z)\le \epsilon^{-\frac{27}{14}-\delta}\quad\hbox{and}\quad \height(z)\le \epsilon^{-\frac{27}{7}-\delta}
$$
respectively.

\subsection{Estimates on integrability exponents}
Let us apply Corollary \ref{cor:q_v} to the action of ${\rm G}=\hbox{SL}_n$ on ${\rm
  X}=\mathbb{A}^n\backslash \{0\}$. In this case, we have $\mathfrak{a}_{V_K\backslash \{v\}}({\rm
  G})=n^2-n$ and $\mathfrak{a}_{V_K\backslash \{v\}}({\rm  X})=n$. Hence, we conclude that
$$
\mathfrak{q}_v(\hbox{SL}_n)\ge 2(n-1).
$$
This estimate is sharp. 


Another example is given by the orthogonal group 
${\rm G}=\hbox{SO}_{d+1}$ acting on the sphere ${\rm X}={\rm S}^d$, discussed in Section 2.1.
Let $p$ be a prime such that $p=1\mod 4$. In this case we have 
$\mathfrak{a}_{V_\mathbb{Q}\backslash \{p\}}({\rm G})= d^2/4$ for even $d\ge 4$,
$\mathfrak{a}_{V_\Q\setminus\{p\}}({\rm G})= (d+1)(d+3)/4$ for odd $d\ge 4$, and
$\mathfrak{a}_{V_\mathbb{Q}\backslash \{p\}}({\rm X})= d-1$. Therefore,
\begin{align*}
&\mathfrak{q}_v(\hbox{SO}_{d+1})\ge \frac{d^2}{2d-2},\quad\hbox{when $d$ is even},\\
&\mathfrak{q}_v(\hbox{SO}_{d+1})\ge \frac{(d+1)(d+3)}{2d-2},\quad\hbox{when $d$ is odd}.
\end{align*}



%

\section{Semisimple groups and spherical functions}
\label{sec:semisimple_and_spherical}

Let $K$ be a number field. We denote by $V_K$ the set of all places of $K$, which consists
of the finite set $V_K^\infty$ of Archimedean places and  the set $V_K^f$ of non-Archimedean places.
For $v\in V_K$, we denote by $|\cdot|_v$ the suitably normalised absolute value and by $K_v$
the corresponding completion.
Also for $v\in V_K^f$ we denote by $O_v\subset K_v$ the ring of integers and by $q_v$ the cardinality
of the residue field.

\subsection{Structure of semisimple groups}
We recall elements of the structure theory of semisimple algebraic groups over local fields,
which is discussed in details in \cite{BT,Tits}.
Let ${\rm G}\subset \hbox{GL}_n$ be a connected semisimple  algebraic group defined over a number field $K$.
For all but finitely many places in $V_K$,
\begin{itemize}
\item[(i)] the group $U_v:={\rm G}({O}_v)$ is a hyperspecial, good maximal compact subgroup of $G_v:={\rm
    G}(K_v)$,
\item[(ii)] the group $\rm G$ is unramified over $K_v$ (that is,
$\rm G$ is quasi-split over $K_v$ and split over an unramified extension of $K_v$).
\end{itemize}
For the other places $v$, we fix a good, special maximal compact subgroup $U_v$ 
of $G_v$. For $S\subset V_K$, we set $U_S=\prod_{v\in S} U_v$.

To simplify notation, for an algebraic group ${\rm C}_v$ defined over $K_v$ we denote by $C_v$
the set ${\rm C}_v(K_v)$ of $K_v$-points in ${\rm C}_v$.

Every subgroup $U_v$ is associated to a minimal parabolic $K_v$-subgroup ${\rm P}_v$ of ${\rm G}_v$, 
that has a decomposition ${\rm P}_v={\rm N}_v{\rm Z}_v$ where ${\rm Z}_v$
is connected and reductive, and ${\rm N}_v$ is the unipotent radical of ${\rm P}_v$.
Let $Z_v^\circ$ be the maximal compact subgroup of $Z_v$.
Then $Z_v/Z_v^\circ$ is a free $\mathbb{Z}$-module of rank equal to
the $K_v$-rank of ${\rm G}$. We set $\mathcal{Z}_v=(Z_v/Z^\circ_v)\otimes \mathbb{R}$
and denote by $\nu_v:Z_v\to \mathcal{Z}_v$ the natural map.
Let ${\rm T}_v$ be the maximal $K_v$-split torus in ${\rm Z}_v$.
Then $T^\circ_v:=T_v\cap Z_v^\circ$ is the maximal compact subgroup of $T_v$
and $T_v/T_v^\circ$ has finite index in $Z_v/Z_v^\circ$.
The group of characters of ${\rm Z}_v$ is denoted by $X^*({\rm Z}_v)$.
For any $\theta\in X^*({\rm Z}_v)$ we associate a linear functional $\chi_\theta$ on $\mathcal{Z}_v$
defined by 
\begin{equation}\label{eq:zv}
|\theta(z)|_v=q_v^{\left<\chi_\theta,\nu_v(z)\right>},\quad z\in Z_v,
\end{equation}
where $q_v=e$ for $v\in V_K^\infty$, and $q_v$ is the order of the residue field for $v\in V_K^f$.

The adjoint action of ${\rm T}_v$ on ${\rm G}$
defines a root system in the dual space $\mathcal{Z}_v^*$.
We fix the ordering on  the root system defined by the parabolic subgroup ${\rm P}_v$.
Let $\mathcal{Z}_v^-$ denote the negative Weyl chamber in $\mathcal{Z}_v$ with respect to this ordering.
The relative Weyl group $\mathcal{W}_v:=N_{ G_v}({T}_v)/Z_{ G_v}({T}_v)$ 
operates on $Z_v$ and on $\mathcal{Z}_v$.  
We choose a $\mathcal{W}_v$-invariant scalar product on $\mathcal{Z}_v$.
Then we regard the root system as a subset of $\mathcal{Z}_v$,
 and denote by $\mathcal{Z}_{v,-}$ the cone consisting of negative linear combinations
of simple roots. Let $Z_v^-=\nu_v^{-1}(\mathcal{Z}_v^-)$ and $Z_{v,-}=\nu_v^{-1}(\mathcal{Z}_{v,-})$.

The group $G_v$ has the {\em Cartan decomposition}
$$
G_v=U_v Z_v^- U_v,
$$
which defines a bijection between the double cosets $U_v\backslash G_v/ U_v$
and $Z_v^-/Z^\circ_v$, and the {\em Iwasawa decomposition}
$$
G_v=N_vZ_v U_v,
$$
which defines a bijection between $N_v\backslash G_v/ U_v$ and $Z_v/Z^\circ_v$.
For $g\in G_v$, we define $z(g)\in Z_v/Z^\circ_v$ by the property $g\in N_v z(g) U_v$.

For unramified $v$, we have $T_v/T_v^\circ=Z_v/Z_v^\circ$ and the Cartan decomposition becomes
\begin{equation}\label{eq:cartan2}
G_v=U_v T_v^- U_v,
\end{equation}
where $T_v^-=\nu_v^{-1}(\mathcal{Z}_v^-)$.

For $S\subset V_K$, we denote by $G_S$ the restricted direct product of the groups
$G_v$ for  $v\in S$ with respect to the family of subgroups $U_v$.
Then $G_S$ is a locally compact second countable group, and $U_S:=\prod_{v\in S} U_v$.
is a maximal compact subgroup of $G_S$.

For every $v\in V_K$, we denote by $m_{G_v}$ the Haar measure on $G_v$ which is normalised so
that $m_{G_v}(U_v)=1$ for non-Archimedean $v$. Then the restricted product measure $\otimes_{v\in S} m_{G_v}$
defines a Haar measure on the group $G_S$ which we denote by $m_{S}$.
We also denote by $m_{U_v}$ the probability Haar measure on $U_v$.

\subsection{Spherical functions}
The theory of spherical functions on semisimple groups over local fields 
was developed by Harish-Chandra (\cite{HC1}, \cite[Chapter IV]{Hel}) in the Archimedean case, and
by Satake (\cite{Sat}, \cite{Car}) in the non-Archimedean case.

For $v\in V_K$, we introduce the {\em Hecke algebra} $\mathcal{H}_v$ 
consisting of bi-$U_v$-invariant continuous functions with compact support on $G_v$
and equipped with product given by convolution. The structure of this algebra can
be completely described, and in particular, it is commutative.

A continuous bi-$U_v$-invariant function $\eta:G_v\to\mathbb{C}$ with compact support
is called {\em spherical} if $\eta(e)=1$ and one of the following equivalent conditions holds:
\begin{enumerate}
\item[(a)]  $\eta$ is bi-$U_v$-invariant, and the map
$$
\mathcal{H}_v\to \mathbb{C}:\phi\mapsto \int_{G_v} \phi(g)\eta(g^{-1})\, dm_{G_v}(g) 
$$
is an algebra homomorphism.
\item[(b)] for every $\phi\in \mathcal{H}_v$ there exists $\lambda_\phi\in\mathbb{C}$ such that $\phi*\eta=\lambda_\phi\eta$.
\item[(c)] For every $g_1,g_2\in G_v$,
$$
\eta(g_1)\eta(g_2)= \int_{U_v} \eta(g_1ug_2)\, dm_{U_v}(u).
$$
\end{enumerate}
 
We shall use the following  parametrisation of the spherical functions as well as some of their basic properties, due to 
Harish Chandra \cite{HC1}\cite{HC2}\cite{HC3} and Satake \cite{Sat}. 


\begin{Theorem}\label{th:satake}
Every spherical function on $G_v$
is of the form
\begin{equation}\label{eq:spher}
\eta_\chi(g)=\int_{U_v} q_v^{\left<\chi, \nu_v(z(ug))\right>} \Delta_{v}^{-1/2}(z(ug))\, dm_{U_v}(u)
\end{equation}
where $\chi\in \mathcal{Z}^*_v\otimes \mathbb{C}$ and $\Delta_{v}$ denotes the modular function of
$P_v$. 

Furthermore, $\eta_{\chi}=\eta_{\chi'}$
if and only if $\chi$ and $\chi'$ are on the same orbit of the Weyl group $\mathcal{W}_v$.
\end{Theorem}

Let $\rho_v\in \mathcal{Z}_v^*$ denotes the character corresponding to the 
half-sum of positive roots. Then
$$
\Delta_v(z)=q_v^{2\left<\rho_v, \nu_v(z)\right>},\quad z\in Z_v.
$$
Let $\Pi_v=\{\alpha\}\subset \mathcal{Z}_v^*$
 denotes the system of simple roots corresponding to the parabolic subgroup ${\rm P}_v$.
We denote by $\{\alpha'\}_{\alpha\in\Pi_v}$ the basis dual to $\Pi_v$.
A character $\chi\in \mathcal{Z}_v^*\otimes \mathbb{C}$ is called {\it dominant} if
$\hbox{Re}\left<\chi,\alpha'\right>\ge 0$ for all $\alpha\in \Pi_v$. 
Note that  every $\chi$ can be conjugated to a dominant one by $\mathcal{W}_v$.
Therefore, in the discussion of spherical functions we may restrict our attention
to dominant $\chi$'s.

We recall the following well-known properties of spherical functions:

\begin{Lemma}\label{l:sph_basic}
For dominant $\chi\in \mathcal{Z}^*_v\otimes \mathbb{C}$,
\begin{itemize}
\item the spherical function $\eta_\chi$ is bounded if and only if
$$
\hbox{\rm  Re}\left<\chi,\alpha'\right>\le \left<\rho,\alpha'\right>
$$
for all $\alpha\in \Pi_v$.
\item the spherical function $\eta_\chi\in L^p(G_v)$ if and only if
$$
{\rm  Re} \left< \chi,\alpha' \right> < (1-1/p)\left<\rho,\alpha'\right>
$$
for all $\alpha\in \Pi_v$.
\end{itemize}
\end{Lemma}


We shall also need the following estimates.

\begin{Lemma}\label{l:sph_est}
\begin{enumerate}
\item[(i)] For dominant $\chi\in \mathcal{Z}_v^*\otimes \mathbb{C}$,
$$
|\eta_\chi(z)|\le q_v^{\hbox{\rm \tiny Re} \left<\chi,\nu_v(g)\right>}
\eta_0(g),\quad g\in Z^-_v.
$$
\item[(ii)] For every $\epsilon>0$,
$$
\eta_0(g)\le c_\epsilon\, \Delta_{v}^{-1/2+\epsilon}(g),\quad g\in Z_v,
$$
with $c_\epsilon$ bounded uniformly in $v$.
\end{enumerate}
\end{Lemma}

\begin{proof}
As to part (i), note that for the Archimedean case, this lemma is proved in \cite[Proposition~7.15]{knapp}.
The proof for non-Archimedean case is similar.
We have
$$
|\eta_\chi(g)|=\int_{U_v}  q_v^{ \hbox{\rm \tiny Re}  \left<\chi,\nu_v(z(ug))\right>}
\Delta_{v}^{-1/2}(z(ug))\, dm_{U_v}(u).
$$
Since the double cosets $U_vgU_v$ and $N_vz(ug)U_v=N_v(ug)U_v$ have nontrivial intersection,
it follows from \cite[Proposition~4.4.4]{BT} that for dominant $\chi$,
$$
\hbox{\rm  Re}  \left<\chi,\nu_v(z(ug))\right>\le  \hbox{\rm  Re}  \left<\chi,\nu_v(z)\right>.
$$
This implies the first claim.

As to part (ii), the bound stated of the Harish Chandra $\Xi$-function (denoted $\eta_0$ here) was established by Harish Chandra in both the Archimedean and non-Archimedean case (see also the discussion in \cite{H}).  The uniformity of the bound, namely the fact that $c_\epsilon$ 
is independent of $v$,  has been observed in \cite[\S 6.2]{GN2} and follows from the proof in \cite[Thm. 4.2.1]{Si}. 
\end{proof}

\subsection{Unitary representations}
The spherical functions arise naturally as matrix coefficients of unitary representations.
We denote by $\hat G_v$ the unitary dual of $G_v$
(i.e., the set of equivalence classes of irreducible unitary representations
of the group $G_v$) and by $\hat G_v^1$ the spherical unitary dual
(i.e., the subset consisting of spherical representations).
An irreducible unitary representation is called spherical if it contains a nonzero $U_v$-invariant vector.
Since the Hecke algebra $\mathcal{H}_v$ is commutative, the subspace of $U_v$-invariant vectors
in an irreducible unitary representation is at most one-dimensional. For $\tau_v\in \hat G_v^1$, we denote by $w_{\tau_v}$ a unit $U_v$-invariant vector.
Then the function
$$
\eta_{\tau_v}(g):=\left<\tau_v(g)w_{\tau_v},w_{\tau_v}\right>,\quad g\in G_v,
$$
is a spherical function on $G_v$. Moreover, different elements of $\hat G_v^1$
give rise to different spherical functions. Therefore, 
$\hat G_v^1$ can be identified with a subset of dominant $\chi\in \mathcal{Z}_v^*\otimes\mathbb{C}$
using Theorem \ref{th:satake}.


More generally, for $S\subset V_K$, we denote by $\hat G_S$ the unitary dual of $G_S$
and by $\hat G_S^1$ the spherical unitary dual (with respect to the subgroup $U_S$).
Every $\tau_S\in \hat G_S$ is a restricted tensor product of 
the form $\otimes'_{v\in S} \tau_v$ where $\tau_v\in \hat G_v$ and $\tau_v$
is spherical for almost all $v$ (see e.g. \cite{f}).

We define the  integrability exponent of a $U_S$-spherical  unitary representation
$\tau_S :G_S\to \mathcal{U}(\mathcal{H})$ as
$$
\mathfrak{q}(\tau_S,U_S)=\inf\left\{ q>0:\, \forall\,\, \hbox{$U$-inv. } w\in \mathcal{H}:\, \left< \tau_S(g)w,w\right>\in L^q(G)
\right\}.
$$

\begin{Prop}\label{p:lp}
Let $\tau_v$ for $v\in S$ be a family consisting of irreducible spherical unitary representations of $G_v$.
Then the integrability exponent of the restricted tensor product representation of $G_S$ satisfies 
$$
\mathfrak{q}\left(\otimes'_{v\in S} \tau_v, U_S\right)\le (1+\sigma_S)\sup_{v\in S} \mathfrak{q}(\tau_v, U_v),
$$
where $\sigma_S$ denotes the exponent of the subset $S$ defined in \eqref{eq:sigma_S}.
\end{Prop}

\begin{proof}
\ignore{Let $S'$ be a cofinite subset of $S$ such that every $v\in S'$ satisfies (i) and (ii), and 
$\tau_v$ is spherical for every $v\in S'$.
Since
$$
\pi=\otimes_{v\in S} \pi_v=\left(\otimes_{v\in S-S'} \pi_v\right)\otimes \left(\otimes'_{v\in S'} \pi_v\right),
$$
it suffices to show that both $\otimes_{v\in S-S'} \pi_v$ and  $\otimes'_{v\in S'} \pi_v$ are $L^{q}$. 
Because $S\backslash S'$ is finite, the first claim is obvious.
Hence, without loss of generality, we may assume that $S'=S$.
}
We have to show that the spherical function $\eta=\prod_{v\in S} \eta_{\tau_v}$
is in $L^{q}(G_S)$ for $q>(1+\sigma_S)p$ where $p>p(\tau_v, U_v)$ for all $v\in S$.
To derive the required estimate we use integration calculated in terms of the Cartan decomposition (\ref{eq:cartan2}).
Recall the volume estimate (see \cite[\S 6.2]{GN2} for a discussion) :
\begin{equation}\label{eq:vol1}
m_{G_v}(U_vtU_v)\le c\,\Delta_v(t),\quad t\in T_v^-,
\end{equation}
which is uniform over $v$.
Let the representations $\tau_v$ correspond to dominant $\chi_v\in\mathcal{Z}_v\otimes\mathbb{C}$.
Since the spherical function $\eta_{\chi_v}$ is in $L^{p}(G_v)$, it follows 
from Lemma \ref{l:sph_basic} that for $t\in T_v^-$,
$$
\left|q_v^{\left<\chi_v,\nu_v(t)\right>}\right|\le \Delta_{v}(t)^{\frac{1}{2}-\frac{1}{p}}.
$$
Hence, Lemma \ref{l:sph_est} implies the estimate
\begin{equation}\label{eq:vol2}
|\eta_{\chi_v} (t)|\le c_\epsilon \Delta_v(t)^{-\frac{1}{p}+\epsilon}
\end{equation}
for every $\epsilon>0$.
Combining (\ref{eq:vol1}) and (\ref{eq:vol2}), we deduce that
\begin{align*}
\int_{G_v} |\eta_{\chi_v} (g)|^q\,dm_{G_v}(g) 
&=\sum_{z\in T_v^-/T^\circ_v} |\eta_{\chi_v}(t)|^q m_{G_v}(U_v t U_v)\\
&\le 1 +\sum_{z\in T_v^-/T_v^\circ-\{e\}} (cc^q_\varepsilon)\Delta_{v}(t)^{-q/p+q\varepsilon+1} \\
&\le 1 +\sum_{z\in T_v^-/ T^\circ_v-\{e\}}
(cc^q_\varepsilon)\left(\prod_{\chi\in\Pi_v} |\chi (t_v)|_v\right)^{-q/p+q\varepsilon+1} \\
&\le 1+\sum_{i_1,\ldots, i_r\in \mathbb{Z}_+, (i_1,\ldots, i_r)\ne
  0}(cc^q_\varepsilon)q_v^{ (-q/p+q\varepsilon+1)\sum_{j=1}^r i_j}\\
&=1+O_\varepsilon\left(q_v^{-q/p+q\varepsilon+1}\right)
\end{align*}
for every $\epsilon>0$. 
It follows from the definition of density $\sigma_S$ that the partial Euler product
$\prod_{v\in S} (1-q_v^{-s})^{-1}$ converges for $s>\sigma_S$. Therefore, we conclude that
$$
\prod_{v\in S} \int_{G_v} |\eta_{\tau_v}|^q\,dm_{G_v}<\infty
$$
provided that $-q/p+1<-\sigma_S$. This completes the proof.
\end{proof}

We denote by $\pi_S$ the unitary representation of $G_S$ on 
$L^2_{00}(G_{V_K}/{\rm G}(K))$. Proposition \ref{p:lp} implies the following estimate
on the integrability exponent of $\pi_S$.

\begin{Cor}\label{cor:unif}
For every $S\subset V_K$,
$$
\mathfrak{q}(\pi_S,U_S)\le \mathfrak{q}_S({\rm G}),
$$
where $\mathfrak{q}_S({\rm G})$ is defined in (\ref{eq:q_s}).
\end{Cor}

Moreover, we also observe that the proof of Proposition \ref{p:lp} gives the following uniform estimate, 
where $\tau\prec\pi_S$ denotes weak containment (see e.g. \cite{chh} for a discussion). 

\begin{Cor}
For every $S\subset V_K$ and $q>\mathfrak{q}_S({\rm G})$,
$$
\sup\left\{\|\eta_\tau\|_q :\, \tau\in \hat{G}_S^1,\, \tau\prec\pi_S\right\}<\infty\,.
$$
\end{Cor}

We now use the  foregoing spectral considerations to obtain an operator norm estimate. First observe 


\begin{Prop}\label{p:norm0}
Let $\beta$ be a bi-$U_S$-invariant finite Borel measure on $G_S$
and $\pi:G_S\to \mathcal{U}(\mathcal{H})$ a strongly continuous unitary representation.
Then
$$
\|\pi(\beta)\|\le \sup\left\{\sqrt{\beta^*(\eta_\tau)\beta(\eta_\tau)}:\, \tau\in \hat{G}_S^1,\, \tau\prec\pi\right\}.
$$
\end{Prop}

\begin{proof}
%
%
To bound $\|\tau(\beta)\|$ for irreducible representations weakly contained in $\pi$, 
we first consider the case when the measure $\beta$ is symmetric.
Let $\mathcal{H}_\tau^{U_S}$ denote the subspace of $\tau(U_S)$-invariant vectors.
It is clear that 
$$
\tau(\beta)\mathcal{H}_\tau\subset\mathcal{H}_\tau^{U_S}.
$$
In particular, $\tau(\beta)=0$ for $\tau\notin \hat{G}_S^1$.
Since $\beta$ is symmetric, the operator $\tau(\beta)$ is self-adjoint, and hence,
$$
\tau(\beta)(\mathcal{H}_\tau^{U_S})^\perp \subset (\mathcal{H}_\tau^{U_S})^\perp.
$$
This implies that
$$
\|\tau(\beta)\|=\int_{G_S} \left<\tau(\beta)w_\tau,w_\tau\right>\, d\beta(g)=\beta(\eta_\tau)
$$
where $w_\tau\in \mathcal{H}_\tau^{U_S}$ with $\|w_\tau\|=1$.
This completes the proof when $\beta$ is symmetric. In general, we have
$$
\|\pi(\beta)\|^2=\|\pi(\beta^* *\beta)\|
$$
Hence, by the previous argument,
$$
\|\pi(\beta)\|\le \sup\left\{\sqrt{(\beta^* *\beta)(\eta_\tau)}:\, \tau\in \hat{G}_S^{1}, \tau\prec\pi\right\}.
$$
Since $\beta$ is bi-$U_S$-invariant, it follows from the properties of spherical functions that
$$
(\beta^* *\beta)(\eta_\tau)=\beta^*(\eta_\tau) \beta(\eta_\tau),
$$ 
which implies the claim.
\end{proof}
We can now obtain the following operator norm estimate.

\begin{Prop}\label{p:norm}
Let $\beta$ be a Haar-uniform probability measure supported on a bi-$U_S$-invariant
bounded subset $B\subset G_S$. 
Then
$$
\|\pi_S(\beta)\|\ll_{S,\delta} m_{S}(B)^{-\frac{1}{\mathfrak{q}_S({\rm G})}+\delta}.
$$
for every $\delta>0$.
\end{Prop}

\begin{proof}
In view of Proposition \ref{p:norm0} we need to establish a uniform estimate 
$$
\beta(\eta_\tau)=\frac{1}{m_{S}(B)}\int_B \eta_\tau(g)\, dm_{S}(g),
$$
where $\eta_\tau$ is the spherical function of a representation $\tau\in \hat G_S^1$
which is weakly contained in $\pi_S$.
By H\"older's inequality, for $q>\mathfrak{q}_S({\rm G})$, we have
$$
\beta(\eta_\tau)\le \frac{1}{m_{S}(B)} \|\chi_B\|_{(1-1/q)^{-1}} \|\eta_\tau\|_q
=m_{S}(B)^{-1/q}\|\eta_\tau\|_q.
$$
Now the claim follows from Proposition \ref{p:norm0} and Corollary \ref{cor:unif}. 
\end{proof}

\section{Mean ergodic theorem}\label{sec:mean}

We keep the notation from the previous section.
In particular, ${\rm G}$ denotes a connected semisimple group defined over a number field $K$.
Our aim is to prove the mean ergodic theorem for the space 
$\Upsilon:=G_{V_K}/{\rm G}(K)$ equipped with the invariant probability measure $\mu$.
For $S\subset V_K$, we consider the natural action of $G_{V_K\backslash S}$ on $\Upsilon$ and,
given a Haar-uniform probability measure $\beta$ on $G_{V_K\backslash S}$, 
the averaging operator
\begin{equation}\label{eq:average}
\pi_{V_K\backslash S}(\beta)\phi(\varsigma)=\int_{G_{V_K\backslash
    S}}\phi(g^{-1}\varsigma)\,d\beta(g), \quad \phi\in L^2(\Upsilon).
\end{equation}

We first consider the case when the group ${\rm G}$ is simply connected.
For simply connected groups, the mean ergodic theorem admits a simple version,
but the general case requires more delicate
considerations because of the presence of nontrivial automorphic characters.

Let $\mathcal{X}_{aut}(G_{V_K})$ be the set of automorphic characters, that is,
the set consisting of continuous unitary characters $\chi$ 
of $G_{V_K}$ such that $\chi({\rm G}(K))=1$.

\begin{Lemma}\label{l:char_sc}
If ${\rm G}$ is simply connected, then $\mathcal{X}_{aut}(G_{V_K})=1$.
\end{Lemma}

\begin{proof}
The group ${\rm G}$ is isotropic over $K_v$ for some $v\in V_K$ (see \cite[Theorem 6.7]{PlaRa}.
Then the group $G_v$ coincides with its commutator (see \cite[\S7.2]{PlaRa}), and hence 
$\chi(G_v)=1$ for every character $\chi$ of $G_{V_K}$.
Since by the strong approximation property ${\rm G}(K)G_v$ is dense in $G_{V_K}$,
the claim follows.
\end{proof}

Lemma \ref{l:char_sc} implies that for simply connected groups, $L_{00}^2(G_{V_K}/{\rm G}(K))$
is the space of functions with zero integral. Hence, Proposition \ref{p:norm} gives

\begin{Theorem}[mean ergodic theorem]\label{cor:mean_simply connected}
Assume that ${\rm G}$ is simply connected, and let $\beta$ be the Haar-uniform probability
measure supported on bi-$U_{V_K\backslash S}$-invariant subset $B$ of $G_{V_K\backslash S}$.
Then for every $\phi\in L^2(\Upsilon)$,
$$
\left\| \pi_{V_K\backslash S}(\beta)\phi-\int_\Upsilon\phi\,d\mu\right\|_2\ll_{S,\delta}
m_{V_K\backslash S}(B)^{-\frac{1}{\mathfrak{q}_{V_K\backslash S}({\rm G})}+\delta}
$$
for every $\delta>0$.
\end{Theorem}

The general case of Theorem  \ref{cor:mean_simply connected} requires two auxiliary lemmas. 

\begin{Lemma}\label{l:normal}
Let $p:\tilde{\rm G}\to {\rm G}$ be a simply connected cover 
of a connected semisimple group defined over $K$. Then for every 
$S\subset V_K$, $p(\tilde G_S)$ is a normal co-Abelian subgroup of $G_S$.
Moreover, if $S$ is finite, then $p(\tilde G_S)$ has finite index in $G_S$.
\end{Lemma}

\begin{proof}
For every $v\in V_K$, we have the exact sequence in Galois cohomology 
$$
\tilde G_v\stackrel{p}{\to} G_v \to H^1(K_v,\hbox{ker}(p)).
$$
Since $ H^1(K_v,\hbox{ker}(p))$ is Abelian, $p(\tilde G_v)$ is a normal co-Abelian subgroup of $G_v$.
This implies that $p(\tilde G_v)$ contains the commutator of $G_v$.
Since $H^1(K_v,\hbox{ker}(p))$ is finite group (see \cite[\S6.4]{PlaRa}),
$p(G_v)$ has finite index in $G_v$.
Also for almost all $v\in V_K$, we have the exact sequence
$$
\tilde{\rm  G}(O_v)\stackrel{p}{\to} {\rm  G}(O_v) \to
H^1\left(K_v^{ur}/K_v,\hbox{ker}(p)(O_v^{ur})\right)
$$
(see \cite[Proposition 6.8]{PlaRa}),
where $K_v^{ur}$ denotes the unramified closure of $K_v$ and
$O_v^{ur}$ denotes its ring of integers. Therefore, $p(\tilde{\rm  G}(O_v))$
is a normal co-Abelian subgroup of ${\rm  G}(O_v)$, and in particular,
$p(\tilde{\rm  G}(O_v))$ contains the commutator of ${\rm  G}(O_v)$.

The group $G_S$ is a union of the subgroups $G_{S,S^\prime}$ defined by 
$$
G_{S,S'}:=\left(\prod_{v\in S'} G_v\right)\left(\prod_{v\in S'} {\rm G}(O_v)\right)\,\,,
$$
as $S'$ runs over finite subsets of $S$.
It follows from the first paragraph that $p(\tilde G_S)$ contains
the commutator of $G_{S,S'}$ for every finite $S'\subset S$.
Hence, $p(\tilde G_S)$ contains the commutator of
$G_S$, and $p(\tilde G_S)$ is a normal co-Abelian subgroup of $G_S$.
\end{proof}

Let $S'$ be a finite subset of $V_K$.
For an open subgroup $U$ of $G_{V^f_K\backslash S'}$, we denote by $\mathcal{X}_{aut}(G_{V_K})^U$ the subset of
$\mathcal{X}_{aut}(G_{V_K})$ consisting of $U$-invariant characters.
Let $G^U$ denote the kernel of $\mathcal{X}_{aut}(G_{V_K})^U$ in $G_{V_K}$.

\begin{Lemma}\label{l:characters}
The group $\mathcal{X}_{aut}(G_{V_K})^U$ is finite, and
$G^U$ has finite index in $G_{V_K}$.
\end{Lemma}

\begin{proof}
Let $p:\tilde {\rm G}\to {\rm G}$ be the simply connected cover.
For every $\chi\in \mathcal{X}_{aut}(G_{V_K})$, we have 
$\chi\circ p \in \mathcal{X}_{aut}(\tilde G_{V_K})$. Hence, it follows from
Lemma \ref{l:char_sc} that $\chi\circ p=1$. We conclude that every $\chi\in
\mathcal{X}_{aut}(G_{V_K})^U$ vanishes on $H:={\rm G}(K)p(\tilde G_{V_K})U$.
By Lemma \ref{l:normal}, $p(\tilde G_{V_K})$ is a normal co-Abelian subgroup of $G_{V_K}$.
Hence it follows that $H$ is a normal co-Abelian subgroup  as well.
We claim that it has finite index in $G_{V_K}$.

For every $v\in V_K$, we have the exact sequence 
$$
\tilde G_v\to G_v\to H^1(K_v,\ker(p)),
$$
where the last term is finite by \cite[\S6.4]{PlaRa}.
This shows that $p(\tilde G_{V_K^\infty \cup S'})$ has finite index
in $G_{V_K^\infty\cup S'}$. The number of double cosets (i.e., the class number)
of the subgroups ${\rm  G}(K)$ and $G_{V_K^\infty}U_{V_K^f\cap S'}U$ in $G_{V_K}$ is finite
(see \cite[\S8.1]{PlaRa}). Then the number of double cosets of ${\rm G}(K)$ and $p(\tilde G_{V_K^\infty\cup S'})U$
in $G_{V_K}$ 
is finite as well. From this we conclude that the number of double cosets
of ${\rm G}(K)$ and $p(\tilde G_{V_K})U$ in $G_{V_K}$ is finite, and since $p(\tilde G_{V_K})$
is co-Abelian,  the factor group $G_{V_K}/H$ is finite.

We have shown above that every $\chi\in \mathcal{X}_{aut}(G_{V_K})^U$
factors through the finite factor group $G_{V_K}/H$.
This implies that $\mathcal{X}_{aut}(G_{V_K})^U$ is finite
and $G^U$ has finite index in $G_{V_K}$, as required.
\end{proof}

\begin{Theorem}[mean ergodic theorem]\label{c:mean}
Let $S$ be a subset of $V_K$ and  $S'$ a finite subset of $S$.
Let $U^0$ be a finite index subgroup of $U_{V_K^f\cap (S\backslash S')}$
and $U=U_{V_K^f\backslash S}U^0$.
Let $B$ be a bounded measurable subset of $G_{V_K\backslash S}\cap G^U$
which is $U_{V_K^f\backslash S}$-biinvariant and
$\beta$ the Haar-uniform probability measure supported on 
the subset $U^0B$ of $G_{(V_K\backslash S)\cup (V^f_K\backslash S')}$.
Then for every $\phi\in L^2(\Upsilon)$ such that 
$\supp(\phi)\subset G^U$, we have
$$
\left\|\pi_{(V_K\backslash S)\cup (V^f_K\backslash S')}(\beta)\phi -\left(\int_\Upsilon \phi\,
    d\mu\right)\xi_U \right\|_2\ll_{U^0, S,\delta}
m_{V_K\backslash S}(B)^{-\frac{1}{\mathfrak{q}_{V_K\backslash S}({\rm G})}+\delta}\|\phi\|_2
$$
for every $\delta>0$, where $\xi_U$ is the function on $\Upsilon$
such that $\xi_U=|G_{V_K}:G^{U}|$ on the open set $G^{U}/{\rm G}(K)\subset \Upsilon$
and $\xi_U=0$ otherwise.
\end{Theorem}

\begin{proof}
We have the decomposition 
$$
L^2(\Upsilon)=\mathcal{H}^0\oplus \mathcal{H}^1\oplus \mathcal{H}^2
$$
where $\mathcal{H}^0$ is the orthogonal complement of $\mathcal{X}_{aut}(G_{V_K})$,
$\mathcal{H}^1$ is the (finite dimensional) span for $\mathcal{X}_{aut}(G_{V_K})^U$, and
$\mathcal{H}^2$ is the orthogonal complement of $\mathcal{H}^0\oplus \mathcal{H}^1$.
For $\phi\in L^2(\Upsilon)$, we have the corresponding decomposition
$$
\phi=\phi_0+\phi_1+\phi_2.
$$
We observe that the set $\mathcal{X}_{aut}(G_{V_K})^U$ can be identified with the group
of characters of the finite Abelian group $G_{V_K}/G^U$. Hence, it follows that
$$
\sum_{\chi\in \mathcal{X}_{aut}(G_{V_K})^U} \chi=\xi_U.
$$
Since $\mathcal{X}_{aut}(G_{V_K})^U$ forms an orthonormal basis  of $\mathcal{H}^1$,
\begin{align*}
\phi_1=\sum_{\chi\in \mathcal{H}^1} \left<\phi, \chi \right>\chi
=\left(\int_\Upsilon \phi\, d\mu\right) \sum_{\chi\in \mathcal{X}_{aut}(G_{V_K})^U} \chi
=\left(\int_\Upsilon \phi\, d\mu\right)\xi_U.
\end{align*}
Since the measure $\beta$ is $U$-invariant, for every $\chi\in \mathcal{X}_{aut}(G_{V_K})$ and
$u\in U$, we have
$$
\pi_{(V_K\backslash S)\cup (V_K^f\backslash S')}(\beta)\chi=
\chi(u)\pi_{(V_K\backslash S)\cup (V_K^f\backslash S')}(\beta)\chi.
$$
Therefore, if $\chi$ is not $U$-invariant, then
$$
\pi_{(V_K\backslash S)\cup (V_K^f\backslash S')}(\beta)\chi=0.
$$
This implies that
$$
\pi_{(V_K\backslash S)\cup (V_K^f\backslash S')}(\beta)|_{\mathcal{H}^2}=0.
$$
Also, since $\supp(\beta)\subset G^U$, we have
$$
\pi_{(V_K\backslash S)\cup (V_K^f\backslash S')}(\beta)|_{\mathcal{H}^1}=id.
$$
We conclude that
$$
\pi_{(V_K\backslash S)\cup (V_K^f\backslash S')}(\beta)\phi -\left(\int_\Upsilon \phi\, d\mu\right)\xi_U=
\pi_{(V_K\backslash S)\cup (V_K^f\backslash S')}(\beta)\phi_0.
$$
By Jensen inequality,
\begin{align*}
\left\|\pi_{(V_K\backslash S)\cup (V_K^f\backslash S')}(\beta)\phi_0\right\|_2
\ll_{U^0} \left\|\pi_{V_K\backslash S}(\beta')\phi_0\right\|_2,
\end{align*}
where $\beta'$ is the Haar-uniform probability measure supported on $B\subset G_{V_K\backslash S}$.
Therefore, by Proposition \ref{p:norm},
$$
\left\|\pi_{(V_K\backslash S)\cup (V_K^f\backslash S')}(\beta)\phi_0\right\|_2
\ll_{U^0,S,\delta} m_{V_K\backslash S}(B)^{-\frac{1}{\mathfrak{q}_{V_K\backslash S}({\rm G})}+\delta}\|\phi_0\|_2
$$
for every $\delta>0$. This implies the claim.
\end{proof}

\section{The duality principle}
\label{sec:duality}
\subsection{Duality and almost sure approximation on the group variety}
In this section we develop an instance of the duality principle in homogeneous spaces, which is analyzed more generally in \cite{GN3}. Heuristically speaking, the principle asserts that given a lattice subgroup $\Gamma$ of $G$, and another closed subgroup $H$,  the effective estimates on the ergodic behavior of $\Gamma$-orbits in $G/H$ can be deduced from the effective estimates on the ergodic behavior of $H$-orbits in $G/\Gamma$. This accounts for the essential role that the automorphic representation and its restriction to subgroups plays in our considerations.

Let ${\rm G}\subset \hbox{GL}_n$ be a linear algebraic group defined over a number field $K$.
In this section we relate the problem of Diophantine approximation on
${\rm G}$ and on its homogeneous spaces with ``shrinking target'' properties
of suitable dynamical systems.
More precisely, we consider the space $\Upsilon=G_{V_K}/{\rm G}(K)$ equipped with
the natural translation action of the group
$G_{V_K\backslash S}$ where $S\subset V_K$. We show that a point in $G_S$
can be approximated by rational points in ${\rm G}(K)$ with given accuracy
provided that the corresponding orbit of $G_{V_K\backslash S}$
in $\Upsilon$ can be used to approximate the identity coset with a suitable accuracy.

Let $\|\cdot\|_v$ be the maximum norm on $K_v^n$, that is, $\|x\|_v=\max_i |x_i|_v$ for $x\in K_v^n$.
We observe that for $g\in \GL_n(K_v)$,
\begin{equation}\label{eq:metric}
\|g\cdot x\|_v\le c_v(g) \|x\|_v,
\end{equation}
where $c_v(g)=\max_{i,j} |g_{ij}|_v$ for $v\in V_K^f$ and $c_v(g)=n \max_{i,j} |g_{ij}|_v$ for $v\in V_K^\infty$.
Note that $c_v(g)$ is bounded on compact subsets of $\GL_n(K_v)$, and 
$c_v(g)=1$ for $g\in\GL_n(O_v)$.

The height function $\height$ on $G_{V_K}$ is defined by
\begin{equation}\label{eq:hhh}
\height(g)=\prod_{v\in V_K} \max_{i,j} (1,|g_{ij}|_v), \quad g\in G_{V_K}.
\end{equation}
This extends the definition of the height function from (\ref{eq:height}).

We note for  future reference that given a bounded $\Omega\subset G_{V_K}$,
\begin{equation}\label{l:height}
\height(b_1 gb_2)\ll_\Omega \height(g)
\end{equation}
for every $b_1,b_2\in\Omega$ and $g\in G_{V_K}$.

For $v\in V_K$ and $\epsilon>0$, we set
$$
\mathcal{O}_v(\epsilon):=\{g\in G_v:\, \|g-e\|_v\le \epsilon\},
$$
and for any $S\subset V_K$,
\begin{equation}\label{eq:o_s}
\mathcal{O}_S(\epsilon):=\prod_{v\in S} \mathcal{O}_v(\epsilon)\footnote{Note that $\mathcal{O}_v(\epsilon)=\mathcal{O}_{\{v\}}(\epsilon)$ and $\mathcal{O}_v(1)=O_v\neq O_{\{v\}}$.}.
\end{equation}

Let $m$ be the Tamagawa measure on $G_{V_K}$. We refer to \cite[\S14]{vosk}
for a detailed construction of $m$. Given a nonzero left-invariant differential
$K$-form of degree $\dim ({\rm G})$ on ${\rm G}$, one defines the corresponding
local measures $m_v$ on $G_v$. We assume that the product 
$\prod_{v\in V_K^f} m_v({\rm G}(O_v))$ converges. 
Then the Tamagawa measure is given by
$$
m=\prod_{v\in V_K} m_v.
$$
We note that when ${\rm G}$ is a semisimple group, then the product 
of $m_v({\rm G}(O_v))$ is known to converge (see \cite[\S14]{vosk}).
This property is crucial for us to prove results on Diophantine approximation
that are uniform over finite subsets $S'$ of $V_K$.

\begin{remark} {\rm
When the product $\prod_{v\in V_K^f} m_v({\rm G}(O_v))$ diverges, then the Tamagawa
measure is defined by $m=\prod_{v\in V_K} \lambda_v m_v$ where $\lambda_v>0$ are
suitably chosen convergence factors. The arguments of this section can be modified
to deal with this case as well. However, one loses uniformity over subsets $S'$.
}
\end{remark}

We use the following estimates for the local measures.
Recall that $I_v=\{q_v^n\}_{n>0}$ for $v\in V_K^f$ and $I_v=(0,1)$ for $v\in V_K^\infty$.

\begin{Lemma}\label{lem:balls}
There exist $d_v\in (0,1]$ such that $d_v=1$ for almost all $v$ and
$$
m_v(\mathcal{O}_v(\epsilon))\ge d_v \epsilon^{r_v\dim ({\rm G})}
$$
for all $\epsilon\in I_v$.
\end{Lemma}

\begin{proof}
The measures $m_v$ in local coordinate are given by
$$
|h(x)|_v (dx_1\wedge\cdots \wedge dx_d)_v,
$$
where $h(x)$ is a convergent power series and $d=\dim({\rm G})$. 
This implies the estimate $m_v(\mathcal{O}_v(\epsilon))\ge d_v \epsilon^{r_v\dim ({\rm G})}$
with some $d_v\in (0,1]$. We note that for $v\in V_K^f$, the neighbourhoods 
$\mathcal{O}_v(\epsilon)$ with $\epsilon=q_v^{-n}$ are precisely the congruence
subgroups 
$$
{\rm G}^{(n)}(O_v)=\{g\in {\rm G}(O_v):\, g=e\mod \mathfrak{p}_v^n\},
$$
where $\mathfrak{p}_v$ denotes the maximal ideal of $O_v$ and $q_v$ denotes the cardinality
of the residue field. The computation in \cite[\S14.1]{vosk} shows that for
almost all $v$ and $n\ge 1$,
$$
m_v({\rm G}^{(n)}(O_v))=q_v^{-nd}.
$$ 
This implies the claim.
\end{proof}

The measure $m$ on $G_{V_K}$ defines the probability  Haar measure on the factor space
$\Upsilon=G_{V_K}/{\rm G}(K)$,
which we denote by $\mu$. The measure $\mu$ is defined so that for measurable subsets $B\subset
G_{V_K}$ that project injectively  on $\Upsilon$,
one has 
$$
\mu(B{\rm G}(K))=\frac{1}{m(G_{V_K}/{\rm G}(K))}m(B).
$$

Now we establish a series of results that provide a connection between dynamics on the space
$\Upsilon$ and Diophantine approximation. The first two results deal with group varieties
and the last two results with general homogeneous varieties.

{\it Convention:} We shall use the product decompositions of the adele group defined by  any subset $Q\subset V_K$ 
$$
G_{V_K}=G_{V_K\backslash Q}\times G_Q\,.
$$
In order to simplify notation we identify a subset $B\subset G_{V_K\backslash Q}$
with the subset $B\times \{e\}$ of $G_{V_K}$.

In Propositions \ref{p:dual} and \ref{p:dual2} below, all implicit constants may depend on 
the set of places $S$ and the bounded subset $\Omega$,
but are independent of other parameters unless stated otherwise.

\begin{Prop}\label{p:dual}
Fix $S\subset V_K$ and a bounded subset $\Omega$ of $G_S$.
Then there exists a family of measurable subsets
$\Phi_\epsilon$ of $\Upsilon$ indexed by $\epsilon=(\epsilon_v)_{v\in S'}$,
where $S'$ is a finite subset of $S$ and $\epsilon_v\in I_v$,
that satisfies
\begin{equation}\label{eq:llow}
\mu(\Phi_\epsilon)\gg \prod_{v\in S'} \epsilon_v^{r_v\dim ({\rm G})}
\end{equation}
and the following property holds:

if for a subset $B\subset G_{V_K\backslash S}$,
$\epsilon=(\epsilon_v)_{v\in S'}$ as above, $x\in \Omega$ and $\varsigma=(e,x^{-1}){\rm G}(K)\in \Upsilon$,
we have 
\begin{equation}\label{eq:inter}
B^{-1}\varsigma\cap \Phi_\epsilon\ne \emptyset,
\end{equation}
then there exists $z\in {\rm G}(K)$ such that
\begin{equation}\label{eq:cl1}
\height(z)\ll \max_{b\in B} \height(b)
\end{equation}
and
\begin{align}\label{eq:cl2}
&\|x_v-z\|_v\le \epsilon_v \quad\quad\hbox{for all $v\in S'$,}\\
&\|x_v-z\|_v\le 1 \quad\quad\hbox{for all $v\in S\backslash S'$.}\nonumber
\end{align}
\end{Prop}

\begin{proof}
Since $\Omega$ is a bounded subset of $G_S$, there exists  finite $R\subset  S$ such that
$$
\Omega\subset G_{R}\times \prod_{v\in S\backslash R} {\rm G}(O_v).
$$
We fix constants $c_v\ge 1$, $v\in S$, such that $c_v\ge \sup_{g\in\Omega} c_v(g_v^{-1})$,
where $c_v(\cdot)$ is the constant given by \eqref{eq:metric}.
We can take $c_v=1$ for $v\in S\backslash R$.

We set $\epsilon_v=1$ for $v\in S\backslash S'$ and
$\delta_v=\epsilon_v/c_v$ for $v\in S$, $\delta_v=1$ for $v\in V_K\backslash S$.
Let 
$$
\tilde \Phi_\epsilon=\prod_{v\in V_K} \mathcal{O}_v(\delta_v)\subset G_{V_K} \quad\hbox{and}\quad
\Phi_\epsilon=\tilde \Phi_\epsilon {\rm G}(K)\subset \Upsilon.
$$
Since $\delta_v\le 1$ for all $v\in V_K^f$, we have
$\mathcal{O}_v(\delta_v)\subset {\rm G}(O_v)$ for every $v\in V_K^f$.
Therefore, if $z\in {\rm G}(K)$ satisfies
$$
\tilde \Phi_\epsilon z\cap \tilde \Phi_\epsilon\ne \emptyset,
$$
then $z\in {\rm G}(O)$, where $O$ denotes the ring of integers of $K$.
Since the image of the diagonal embedding of $O$ in $\prod_{v\in V_K^\infty} K_v$
is discrete,  there exists $\epsilon_0>0$ (depending only on the number field $K$) such that
$$
\mathcal{O}_{V_K^\infty}(\epsilon_0)z\cap
\mathcal{O}_{V_K^\infty}(\epsilon_0)=\emptyset\quad
\hbox{for every $z\in {\rm G}(O)$, $z\ne e$.}
$$
Then when $\epsilon_v\in (0,\epsilon_0)$ for all $v\in V_K^\infty$,  we have
\begin{equation}\label{eq:z12}
\tilde \Phi_\epsilon z\cap \tilde \Phi_\epsilon =\emptyset\quad
\hbox{for every $z\in {\rm G}(K)$, $z\ne e$.}
\end{equation}
This implies that for such $\epsilon$,
\begin{align*}
\mu(\Phi_\epsilon)&\gg m(\tilde \Phi_\epsilon)=\prod_{v\in V_K} m_v(\mathcal{O}_v(\delta_v))\\
&=\left(\prod_{v\in S'\cup R} m_v(\mathcal{O}_v(\epsilon_v/c_v))\right) \left(\prod_{v\in V_K\backslash S'} m_v(\mathcal{O}_v(1))\right).
\end{align*}
Since the product $\prod_{v\in V_K^f} m_v({\rm G}(O_v))$ converges, we deduce 
using Lemma \ref{lem:balls} that
\begin{align*}
\mu(\Phi_\epsilon)\gg \prod_{v\in S'\cup R} m_v(\mathcal{O}_v(\epsilon_v/c_v)) \gg \prod_{v\in S'} d_v(\epsilon_v/c_v)^{r_v\dim ({\rm G})}
\gg \prod_{v\in S'} \epsilon_v^{r_v\dim ({\rm G})}.
\end{align*}
Since $c_v=1$ for $v\notin R$, the implied constants here are independent of $S'$ (and
depend only on $\Omega$).  This proves estimate (\ref{eq:llow}) provided
that $\epsilon_v \in (0,\epsilon_0)$ for all $v\in V_K^\infty$.

To prove this estimate in general, we set $\epsilon_v'=\epsilon_0\epsilon_v\le \epsilon_0$
for $v\in V_K^\infty$ and $\epsilon_v'=\epsilon_v$ for $v\in V_K^f$.
Then 
\begin{align*}
\mu(\Phi_\epsilon)&\ge \mu(\Phi_{\epsilon'})\gg
\prod_{v\in S'} (\epsilon'_v)^{r_v\dim ({\rm G})}\\
&\ge \left(\prod_{v\in S'\cap V_K^\infty} \epsilon_0^{r_v\dim ({\rm G})}\right) \prod_{v\in S'}
\epsilon_v^{r_v\dim ({\rm G})}.
\end{align*}
This completes the proof of (\ref{eq:llow}).

Now we suppose that (\ref{eq:inter}) holds. Then for some $b\in B$ and $z\in {\rm G}(K)$,
$$
(b^{-1}z,x^{-1}z)\in \tilde \Phi_\epsilon.
$$ 
Hence, for  $v\in S$, we have $x_v^{-1}z\in \mathcal{O}_v(\delta_v)$ and
$$
\|x_v-z\|_v\le c_v\, \|x_v^{-1}z-e\|_v\le c_v\delta_v.
$$
This proves (\ref{eq:cl2}).

Finally, we observe that
$$
z\in  \Omega (b, e)\tilde \Phi_\epsilon.
$$
Hence, it follows from (\ref{l:height}) that
$$
\height(z)\ll \height(b).
$$
This shows (\ref{eq:cl1}) and completes the proof of the proposition.
\end{proof}

\subsection{Duality and approximation at every points on the group variety}

In order to prove the uniform version of our main theorem (Theorem \ref{th:g1}(ii)),
condition (\ref{eq:inter}) in Proposition \ref{p:dual}
needs to be relaxed.  This is achieved by the following proposition:

\begin{Prop}\label{p:dual2}
Fix $S\subset V_K$ and a bounded subset $\Omega$ of $G_S$.
Then there exists a family of measurable subsets
$\tilde \Psi_\epsilon$ of $G_{V_K}$ indexed by $\epsilon=(\epsilon_v)_{v\in S'}$,
where $S'$ is a finite subset of $S$ and $\epsilon_v\in I_v$,
that satisfies
\begin{align}
\mu(\tilde \Psi_\epsilon {\rm G}(K))&\gg \prod_{v\in S'} \epsilon_v^{r_v\dim
  (G)}, \label{eq:llow2}\\
\mu(\tilde \Psi_\epsilon^{-1} (e,x^{-1}){\rm G}(K))&\gg \prod_{v\in S'} \epsilon_v^{r_v\dim
  (G)}\quad\hbox{for all $x\in \Omega$,} \label{eq:llow22}
\end{align}
 and the following property holds:

if for $B\subset G_{V_K\backslash S}$, $\epsilon=(\epsilon_v)_{v\in S'}$ as above, $x\in \Omega$ and
$\varsigma=(e,x^{-1}){\rm G}(K)\in \Upsilon$, we have 
\begin{equation}\label{eq:inter2}
B^{-1} \tilde \Psi^{-1}_\epsilon \varsigma\cap \tilde \Psi_\epsilon {\rm G}(K)\ne \emptyset,
\end{equation}
then there exists $z\in {\rm G}(K)$ such that
\begin{equation*}
\height(z)\ll \max_{b\in B} \height(b)
\end{equation*}
and
\begin{align}\label{eq:dioph2}
&\|x_v-z\|_v\le \epsilon_v\quad\hbox{for all $v\in S'$,}\\
&\|x_v-z\|_v\le 1\quad\hbox{for all $v\in S\backslash S'$}.\nonumber
\end{align}
\end{Prop}

\begin{proof}
Similarly to (\ref{eq:metric}), for every $g\in \hbox{GL}_n(K_v)$ and $y\in\hbox{M}_n(K_v)$,
\begin{equation}\label{eq:norm110}
\|g\cdot y \cdot g^{-1}\|_v\le c_v'(g)\|y\|_v
\end{equation}
where $c_v'(g)\ge 1$, $c_v'(g)=1$ for $g\in \hbox{GL}_n(O_v)$, and $c_v'(g)$ is 
uniformly bounded over bounded subsets. We set $c_v'=\sup_{g\in\Omega} c_v'(g)$.

Let $c_v\ge 1$ be defined as in  Proposition \ref{p:dual}
and $\epsilon_v=1$ for $v\in S\backslash S'$.
We set $\delta_v=1$ for $v\in V_K\backslash S$,
$\delta_v=\epsilon_v/((2n+1)c_vc_v')$ for Archimedean $v\in S$, and
$\delta_v=\epsilon_v/(c_vc_v')$ for non-Archimedean $v\in S$.
Note $\delta_v=1$ for almost all $v$.
Let
$$
\tilde \Psi_\epsilon=\prod_{v\in V_K} \mathcal{O}_v(\delta_v).
$$
Since $\delta_v=1$ for all non-Archimedean $v\notin S'\cup R$,
estimate (\ref{eq:llow2}) can be established by the same argument
as in the proof of  Proposition \ref{p:dual}. To prove (\ref{eq:llow22}), we observe that
$$
\mu(\tilde \Psi_\epsilon^{-1} (e,x^{-1}){\rm G}(K))=\mu((e,x)\tilde \Psi_\epsilon^{-1} (e,x^{-1}){\rm G}(K)),
$$
and by (\ref{eq:norm110}),  
$$
(e,x)\tilde \Psi_\epsilon^{-1} (e,x^{-1})\supset \prod_{v\in V_K} \mathcal{O}_v(\delta_v/c_v')^{-1}.
$$
Hence, estimate (\ref{eq:llow22}) can be proved similarly to (\ref{eq:llow2}).

Suppose that (\ref{eq:inter2}) holds. Then for some $b\in B$, $f'\in \prod_{v\in V_K\backslash S}
\mathcal{O}_v(\delta_v)$, $f\in \prod_{v\in S}
\mathcal{O}_v(\delta_v)$, and $z\in {\rm G}(K)$, we have
$$
(b^{-1}(f')^{-1}z,f^{-1}x^{-1}z)\in \tilde \Psi_\epsilon.
$$ 
For $v\in S\cap V_K^\infty$, we have $f_v^{-1}x_v^{-1}z\in \mathcal{O}_v(\delta_v)$ and 
$f_v\in \mathcal{O}_v(\delta_v)$. We observe that $\|f_v\|_v\le 2$, and 
for every $a\in \hbox{M}_n(K_v)$,
$$
\|f_va\|_v\le n\|f_v\|_v\|a\|_v\le (2n)\|a\|_v.
$$
Hence,
\begin{align*}
\|x_v-z\|_v&\le \|x_v-x_vf_v\|_v+\|x_vf_v-z\|_v\\
&\le c_v\,(\|e-f_v\|_v+ (2n)\|e-f_v^{-1}x_v^{-1}z\|_v)\le (2n+1)c_v\delta_v.
\end{align*}

The case $v\in S\cap V_K^f$ is treated similarly. We observe that $\|f_v\|_v\le 1$, and 
for every $a\in \hbox{M}_n(K_v)$,
$$
\|f_va\|_v\le \|f_v\|_v\|a\|_v\le \|a\|_v,
$$
so that
\begin{align*}
\|x_v-z\|_v&\le \max\{\|x_v-x_vf_v\|_v,\|x_vf_v-z\|_v\}\\
&\le c_v\,\max\{\|e-f_v\|_v,\|e-f_v^{-1}x_v^{-1}z\|_v\}\le c_v\delta_v.
\end{align*}
This  proves (\ref{eq:dioph2}).

Finally, we observe that
$$
z\in  \Omega(f',f)(b, e)\tilde \Psi_\epsilon.
$$
Hence, it follows from (\ref{l:height}) that
$$
\height(z)\ll \height(b).
$$
This completes the proof of the proposition.
\end{proof}

\subsection{Duality and approximation on homogeneous varieties}
Consider now  the problem of Diophantine approximation on general homogeneous varieties.
The results obtained in the present subsection will be used in the proof of Theorem \ref{th:g3}.

In Propositions \ref{p:dual1_prime} and \ref{p:dual2_prime} below,
all implicit constants may depend on $S'\subset S\subset V_K$, $x^0\in X_{S'}$, and $\Omega\subset G_{S'}$,
but are independent of other parameters unless stated otherwise.

\begin{Prop}\label{p:dual1_prime}
Let ${\rm X}\subset \mathbb{A}^n$ be a homogeneous quasi-affine variety of the group ${\rm G}\subset {\rm
  GL}_n$.
Fix $S\subset V_K$, finite $S'\subset S$, $x^0\in X_{S'}$, and a bounded subset $\Omega$ of $G_{S'}$.
Then there exist $\epsilon_0\in (0,1)$ and
 a family of measurable subset $\Phi_\epsilon$ of $\Upsilon$ indexed by
$\epsilon=(\epsilon_v)_{v\in S'}$ where $\epsilon_v\in I_v\cap (0,\epsilon_0)$ that satisfy
\begin{equation}\label{eq:low1}
\mu(\Phi_\epsilon)\gg \prod_{v\in S'} \epsilon_v^{r_v \dim ({\rm X})}
\end{equation}
and the following property holds:

if for $B\subset G_{V_K\backslash S}\times \prod_{v\in V_K^f\cap (S\backslash S')} {\rm G}(O_v)$,
$\epsilon=(\epsilon_v)_{v\in S'}$ as above, and
$\varsigma=(e,g^{-1}){\rm G}(K)\in \Upsilon$ with $g\in \Omega$,  we have
\begin{equation}\label{eq:low4}
B^{-1}\varsigma\cap \Phi_\epsilon\ne \emptyset,
\end{equation}
then there exists $z\in {\rm G}(O_{(V_K\backslash S)\cup S'})$ such that
\begin{equation}\label{eq:low2}
\height(z)\ll \max_{b\in B} \height(b)
\end{equation}
and for $x=gx^0\in X_{S'}$
\begin{equation}\label{eq:low3}
\|x_v-zx^0_v\|\le \epsilon_v\quad\hbox{for all $v\in S'$.}
\end{equation}
\end{Prop}

\begin{proof}
For $v\in S'$, let $p_v: G_{v}\to G_{v}x_v^0$ be the map $g\mapsto g x_v^0$.
There exists a neighbourhood $\mathcal{U}_v$ of $x_v^0$ in $G_{v}x_v^0$
and an analytic section $\sigma_v:\mathcal{U}_v \to G_{v}$ of the map $p_v$ such that $\sigma_v(x_v^0)=e$.
(The section $\sigma_v$ can be constructed using the exponential map in $G_v$.)
Let
$$
\mathcal{O}_{v}(x_v^0,\epsilon):=\{y\in G_{v}x_v^0:\, \|y-x^0_v\|_v\le\epsilon\}.
$$
Let $L_v$ be the stabiliser of $x_v^0$ in $G_{v}$,
and let $L_v^1$ be a compact neighbourhood of identity in $L_v$.
Let $c_v=\sup_{g\in\Omega} c_v(g_v^{-1})$,
where $c_v(\cdot)$ is the constant given by \eqref{eq:metric}.
We set
\begin{align*}
\tilde \Phi_\epsilon= \left( \prod_{v\in V_K\backslash S'} \mathcal{O}_v\right)\times
\left(\prod_{v\in S'} \sigma_v(\mathcal{O}_{v}(x_v^0,\epsilon_v/c_v))L_v^1\right),
\end{align*}
where $\epsilon_v>0$ are sufficiently small, so that $\mathcal{O}_{v}(x_v^0,\epsilon_v/c_v)\subset
\mathcal{U}_v\cap \sigma_v^{-1}({\rm G}(O_v))$, and where $\mathcal{O}_v={\rm G}(O_v)$ for $v\in V_K^f\backslash S'$ and $\mathcal{O}_v$ is
a fixed neighbourhood of identity in $G_v$ for $v\in V_K^\infty\backslash S'$.
We also set
$$
\Phi_\epsilon=\tilde \Phi_\epsilon {\rm G}(K)\subset \Upsilon.
$$
As in the proof of Proposition \ref{p:dual}, we can choose sufficiently small
$\epsilon_v$, $L^1_v$ and $\mathcal{O}_v$, so that
$$
\tilde \Phi_\epsilon z\cap \tilde \Phi_\epsilon =\emptyset\quad\hbox{for every $z\in {\rm
    G}(K)$, $z\ne e$.}
$$
Then
$$
\mu(\Phi_\epsilon)\gg m(\tilde \Phi_\epsilon)\gg \prod_{v\in S'} m_v(\sigma_v(\mathcal{O}_{v}(x_v^0,\epsilon_v/c_v))L_v^1),
$$
and estimating volumes in local coordinates, we obtain
$$
\mu(\Phi_\epsilon) \gg \prod_{v\in S'} \epsilon_v^{r_v\dim({\rm X})}.
$$
This proves (\ref{eq:low1}).

Now suppose that (\ref{eq:low4}) holds. Then for some $b\in B$ and $z\in {\rm G}(K)$,
we have
$$
(b^{-1}z,g^{-1}z)\in \tilde \Phi_\epsilon.
$$
Then for $v\in S'$, we have $g_v^{-1}z\in \sigma_v(\mathcal{O}_{v}(x_v^0,\epsilon_v/c_v))L_v^1$, and
$$
\|g_v^{-1}z x_v^0-x_v^0\|_v\le \epsilon_v/c_v.
$$
Therefore,
$$
\|z x_v^0-x_v\|_v\le \epsilon_v\quad\hbox{for all $v\in S'$. }
$$
Also, for $v\in V_K^f\cap (S\backslash S')$, we have $z\in b_v\mathcal{O}_v$.
This implies that $z\in {\rm G}(O_{(V_K\backslash S)\cup S'})$.
Hence, we have established (\ref{eq:low3}).

Finally, we observe that since
$$
z\in \Omega(b,e)\tilde \Phi_\epsilon,
$$
estimate (\ref{eq:low2}) follows from (\ref{l:height}).
\end{proof}

We now turn to proving the estimate crucial in establishing diophantine approximation at every point 
in a homogeneous variety. 

\begin{Prop}\label{p:dual2_prime}
Let ${\rm X}\subset\mathbb{A}^n$ be a homogeneous quasi-affine variety of the group ${\rm G}\subset {\rm GL}_n$.
Fix $S\subset V_K$, finite $S'\subset S$, $x^0\in X_{S'}$, and a bounded subset $\Omega$ of $G_{S'}$.
Then there exist $\epsilon_0\in (0,1)$ and
 a family of measurable subset $\tilde\Psi_\epsilon$ of $G_{V_K}$ indexed by
$\epsilon=(\epsilon_v)_{v\in S'}$ where $\epsilon_v\in I_v\cap (0,\epsilon_0)$ that satisfy
\begin{align}
\mu(\tilde \Psi_\epsilon {\rm G}(K))&\gg \prod_{v\in S'} \epsilon_v^{r_v \dim ({\rm X})},\label{eq:low1_1}\\
\mu(\tilde \Psi^{-1}_\epsilon (e,g^{-1}){\rm G}(K))&\gg \prod_{v\in S'} \epsilon_v^{r_v \dim ({\rm X})}
\quad \hbox{for all $g\in\Omega$,}\label{eq:low1_12}
\end{align}
and the following property holds:

if for $B\subset G_{V_K\backslash S}\times \prod_{v\in V_K^f\cap (S\backslash S')} {\rm G}(O_v)$,
$\epsilon=(\epsilon_v)_{v\in S'}$ as above, and
$\varsigma=(e,g^{-1}){\rm G}(K)\in \Upsilon$ with $g\in \Omega$,  we have
\begin{equation}\label{eq:bbb}
B^{-1}\tilde \Psi^{-1}_\epsilon \varsigma\cap \tilde \Psi_\epsilon {\rm G}(K)\ne \emptyset,
\end{equation}
then there exists $z\in {\rm G}(O_{(V_K\backslash S)\cup S'})$ such that
\begin{equation}\label{eq:low2_1}
\height(z)\ll \max_{b\in B} \height(b)
\end{equation}
and for $x=gx_0\in X_{S'}$
\begin{equation}\label{eq:low3_1}
\|x_v-zx^0_v\|\le \epsilon_v\quad\hbox{for all $v\in S'$.}
\end{equation}
\end{Prop}

\begin{proof}
We use notation introduced in the proof of Proposition \ref{p:dual1_prime}.
There exists $c''_v>0$ such that for every $r_v\in \sigma_v(\mathcal{O}_{v}(x_v^0,\epsilon_0))L_v^1$,
we have
$$
\|r_v x\|_v\le c''_v\|x\|_v, \quad x\in K_v^n.
$$
Let $c_v'=c_v(1+c''_v)$  and
\begin{align*}
\tilde \Psi_\epsilon= \left( \prod_{v\in V_K\backslash S'} \mathcal{O}_v\right)\times
\left(\prod_{v\in S'} \sigma_v(\mathcal{O}_{v}(x_v^0,\epsilon_v/c_v'))L_v^1\right),
\end{align*}
Estimate (\ref{eq:low1_1}) is proved as in Proposition \ref{p:dual1_prime}. 
Taking $\mathcal{O}_v$, $\epsilon_v$, and $L_v^1$ sufficiently small,
we can arrange that
$$
(e,g) \tilde \Psi^{-1}_\epsilon (e,g^{-1})z \cap (e,g) \tilde \Psi^{-1}_\epsilon (e,g^{-1})=\emptyset
\quad \hbox{for $z\in {\rm G}(K)$, $z\ne e$.}
$$
Then
$$
\mu(\tilde \Psi^{-1}_\epsilon (e,g^{-1}){\rm G}(K))\gg m(\tilde \Psi^{-1}_\epsilon (e,g^{-1}))=m(\tilde \Psi_\epsilon^{-1}).
$$
Hence, estimate (\ref{eq:low1_12}) follows from (\ref{eq:low1_1}).

Suppose that (\ref{eq:bbb}) holds. 
Then for some $b\in B$, $f'\in \prod_{v\in V_K\backslash S'} \mathcal{O}_v$,
$f\in \prod_{v\in S'} \sigma_v(\mathcal{O}_{v}(x_v^0,\epsilon_v/c_v'))L_v^1$, and $z\in {\rm G}(K)$, we have
$$
(b^{-1}(f')^{-1}z,f^{-1}g^{-1}z)\in \tilde \Psi_\epsilon.
$$
For $v\in S'$,
$$
\|f_vx_v^0-x_v^0\|_v\le \epsilon_v/c'_v\quad\hbox{and}\quad
\|f_v^{-1}g_v^{-1}z x_v^0-x_v^0\|_v\le \epsilon_v/c'_v.
$$
Hence, 
\begin{align*}
\|x_v-z x_v^0\|_v&\le \|g_v x^0_v-g_vf_vx_v^0\|_v+\|g_vf_vx_v^0-z x_v^0\|_v\\
&\le c_v\,(\|x_v^0-f_vx_v^0\|_v+ c''_v\|x_v^0-f_v^{-1}g_v^{-1}z x_v^0\|_v)\le \epsilon_v.
\end{align*}
Also, for $v\in (S\backslash S')\cap V_K^f$,
$$
z\in f'_vb_v{\rm G}({O}_v)\subset {\rm G}({O}_v).
$$
Hence, $z \in {\rm G}(O_{(V_K\backslash S)\cup S'})$.
This completes the proof of (\ref{eq:low3_1}).

Finally, since 
$$
z\in \Omega(f',f)(b,e) \tilde \Psi_\epsilon,
$$
estimate (\ref{eq:low2_1}) follows from (\ref{l:height}).
\end{proof}

\section{Proof of the main results}

\subsection{Almost sure diophantine approximation on the group variety}
In this section we complete the proofs of main results.
Our strategy is to combine the duality principle
from Section \ref{sec:duality} with the mean ergodic theorem from Section \ref{sec:mean}.

Let $S\subset V_K$ and ${\rm G}$ be a connected almost simple algebraic $K$-group.
Our first task is to establish an estimate on  the exponent $\mathfrak{a}_S({\rm G})$ of the variety $G$ 
defined after (\ref{eq:ass_omega}). It is worth mentioning that $\mathfrak{a}_S({\rm G})$
can, in principle, be computed in terms of the root data of ${\rm G}$ using methods from
\cite{GN1,GN2}, but for the purpose of this paper the simple estimate of Lemma \ref{l:ass} below will be sufficient.
We recall that
$$
\mathfrak{a}_S({\rm G}) = \sup_{\Omega\subset G_S}\limsup_{h\to \infty} \frac{\log A_S(\Omega,h)}{\log h},
$$
where $\Omega$ runs over bounded subset of $G_S$, and
$$
A_S(\Omega,h)=|\{z\in {\rm G}(K):\, \height(z)\le h,\;\; z\in \Omega \}|.
$$
We show that  $\mathfrak{a}_S({\rm G})$ can be estimated in terms of volumes of the sets
\begin{equation}\label{eq:b__h}
B_h=U_{V_K\backslash S}\{g\in G_{V_K\backslash S}:\, \height(g)\le h\}U_{V_K\backslash S}\subset
G_{V_K\backslash S}.
\end{equation}

\begin{Lemma}\label{l:ass}
Let $\Omega$ be a bounded subset of $G_S$.
Then there exist $c>1$ and $h_0>0$ such that for every $h\ge h_0$
$$
A_S(\Omega,h)\ll_\Omega m_{V_K\backslash S}(B_{ch}).
$$
In particular, it follows that for every $\delta>0$ and $h\ge h_0(S,\delta)$.
$$
m_{V_K\backslash S}(B_{h})\gg h^{\mathfrak{a}_S({\rm G})-\delta}.
$$
\end{Lemma}

\begin{proof}
Let
$$
\mathcal{A}_S(\Omega,h)=\{\gamma\in {\rm G}(K)\cap G_{V_K\backslash S}\Omega:\, \height(\gamma)\le h  \}.
$$
Then $A_S(\Omega,h)=|\mathcal{A}_S(\Omega,h)|$.
Since ${\rm G}(K)$ is a discrete subgroup of $G_{V_K}$, there exists a bounded neighbourhood $\mathcal{O}$
of identity in $G_{V_K}$ such that $\mathcal{O}\gamma_1\cap \mathcal{O}\gamma_2=\emptyset$
for $\gamma_1\ne \gamma_2\in {\rm G}(K)$. Then
$$
A_S(\Omega,h)=|\mathcal{A}_S(\Omega,h)|\le \frac{m_{V_K}(\mathcal{O}\mathcal{A}_S(\Omega,h))}{m_{V_K}(\mathcal{O})}.
$$
We observe that it follows from (\ref{l:height}) that
$$
\mathcal{O}\mathcal{A}_S(\Omega,h)\subset \mathcal{A}_S(\Omega',c_1 h)\subset
B_{c_2 h}\Omega'.
$$
for a bounded $\Omega'\subset G_S$ and $c_1,c_2>0$. This implies the claim.
\end{proof}

We will also need volume estimates for the intersections of $B_h$ with finite index subgroups.

\begin{Lemma}\label{l:finite}
Let $G_0$ be a finite index subgroup of $G_{V_K}$.
Then there exists $c\ge 1$ such that
$$
m_{V_K\backslash S}(B_h\cap G_0)\ge \frac{m_{V_K\backslash S}(B_{c^{-1}h})}{|G_{V_K}:G_0|}
$$
for every $h>0$.
\end{Lemma}

\begin{proof}
Let $\{g_i\}$ be a finite set of coset representatives of $G_{V_K\backslash S}/(G_{V_K\backslash S}\cap G_0)$.
It follows from (\ref{l:height}) that there exists $c\ge 1$ such that
$g_i^{-1}B_h\subset B_{ch}$. Hence,
\begin{align*}
m_{V_K\backslash S}(B_h)&=\sum_i m_{V_K\backslash S}(B_h\cap g_i G_0)=\sum_i m_{V_K\backslash S}(g_i^{-1}B_h\cap G_0)\\
&\le |G_{V_K}:G_0| m_{V_K\backslash S}(B_{ch}\cap G_0).
\end{align*}
This implies the claim.
\end{proof}

{\it Convention:} In the proofs of the Theorems \ref{th:g1} and \ref{th:g2}, implicit constants may depend
on the set of places $S$, but are independent of the subset $S'$ unless stated otherwise.
In the proof of Theorem \ref{th:g3}, implicit constants may depend on $S'$ as well.

\begin{proof}[Proof of Theorem \ref{th:g1}(i)]
Let $\beta_h$ be the Haar-uniform probability measure supported on the set $B_h$
defined in (\ref{eq:b__h}).
We denote by $\pi_{V_K\backslash S}(\beta_h)$ the corresponding averaging operator acting on
$L^2(\Upsilon)$
(see (\ref{eq:average})).
The main idea of the proof is to combine 
Theorem \ref{cor:mean_simply connected}
with Proposition \ref{p:dual}. By Theorem \ref{cor:mean_simply connected},
for every $\phi\in L^2(\Upsilon)$ and $h,\delta'>0$,
$$
\left\| \pi_{V_K\backslash S}(\beta_h)\phi -\int_\Upsilon \phi\,d\mu\right\|_2
\ll_{\delta'} m_{V_K\backslash S}(B_h)^{-\frac{1}{\mathfrak{q}_{V_K\backslash S}({\rm G})}+\delta'}\|\phi\|_2.
$$
Moreover, it follows from Lemma \ref{l:ass} that this estimate can be rewritten as 
\begin{equation}\label{eq:norm}
\left\| \pi_{V_K\backslash S}(\beta_h)\phi -\int_\Upsilon \phi\,d\mu\right\|_2
\ll_{\delta'} h^{-\frac{\mathfrak{a}_S({\rm G})}{\mathfrak{q}_{V_K\backslash S}({\rm G})}+\delta'}\|\phi\|_2
\end{equation}
for every $\delta'>0$ and $h\ge h_0(\delta')$.

We pick a bounded subset $\Omega$ of $G_S$.
Let $\Phi_\epsilon\subset \Upsilon$ be the sets introduced in Proposition \ref{p:dual}
which are indexed by $\epsilon=(\epsilon_v)_{v\in S'}$ with finite $S'\subset S$ and $\epsilon_v\in I_v$.
We denote by $\phi_{\epsilon}$ the characteristic function of $\Phi_{\epsilon}$. Then (\ref{eq:norm}) gives
\begin{equation}\label{eq:norm2}
\left\| \pi_{V_K\backslash S}(\beta_h)\phi_\epsilon -\mu(\Phi_\epsilon)\right\|_2
\ll_{\delta'} h^{-\frac{\mathfrak{a}_S({\rm G})}{\mathfrak{q}_{V_K\backslash S}({\rm G})}+\delta'}\mu(\Phi_\epsilon)^{1/2}
\end{equation}
for every $\delta'>0$ and $h\ge h_0(\delta')$.
We set 
\begin{equation}\label{eq:h}
h_\epsilon=\left( \prod_{v\in S'} \epsilon_v^{-r_v \frac{\dim({\rm G})}{\mathfrak{a}_S({\rm G})}-\delta}
\right)^{\mathfrak{q}_{V_K\backslash S}({\rm G})/2}
\end{equation}
with $\delta>0$. Let
\begin{equation}\label{eq:u_epsilon}
\Upsilon_\epsilon=\{\varsigma\in\Upsilon:\, B_{h_\epsilon}^{-1}\varsigma\cap \Phi_{\epsilon}=\emptyset\}.
\end{equation}
Then for $\varsigma\in\Upsilon_\epsilon$, we have $\pi_{V_K\backslash S}(\beta_{h_\epsilon})\phi_\epsilon(\varsigma)=0$.
Hence, it follows from (\ref{eq:norm2}) that for every $\delta'>0$, and $\epsilon$, 
we have
\begin{align}\label{eq:pos}
\mu(\Upsilon_\epsilon)&\ll_{\delta'} h_\epsilon^{-\frac{2\mathfrak{a}_S({\rm G})}{\mathfrak{q}_{V_K\backslash S}({\rm G})}+2\delta'}\mu(\Phi_\epsilon)^{-1}
\end{align}
provided that $h_\epsilon\ge h_0(\delta')$.
Combining this estimate with  \eqref{eq:h} and \eqref{eq:llow}, we conclude that
\begin{align}\label{eq:pos1}
\mu(\Upsilon_\epsilon)&\ll_{\Omega,\delta'} \prod_{v\in S'} \epsilon_v^{\theta_v}
\end{align}
for every $\delta'>0$ and $\epsilon$ satisfying $h_\epsilon\ge h_0(\delta')$, where 
\begin{align*}
\theta_v &=\left(-r_v \frac{\dim({\rm G})}{\mathfrak{a}_S({\rm G})}-\delta\right)
\frac{\mathfrak{q}_{V_K\backslash S}({\rm    G})}{2}
\left(-\frac{2\mathfrak{a}_S({\rm G})}{\mathfrak{q}_{V_K\backslash S}({\rm G})}+2\delta'\right)
-r_v\dim({\rm G})\\
&=\mathfrak{a}_S({\rm G}) \delta-\delta'\cdot \mathfrak{q}_{V_K\backslash S}({\rm G})\left(r_v \frac{\dim({\rm G})}{\mathfrak{a}_S({\rm G})}+\delta\right).
\end{align*}
We pick $\delta'=\delta'(\delta)>0$ so that $\theta_v>0$.

Let $E_v=I_v$ for $v\in V_K^f$, and let $E_v=\{2^{-n}\}_{n\ge 1}$ for $v\in V_K^\infty$, where we discretize the parameter in the archemedean part in order  to take advantage of the Borel-Cantelli lemma. 
Set $E_{S'}=\prod_{v\in S'} E_v$ for finite $S'\subset S$.
We denote by $\Upsilon_{S'}$ the $\limsup$ of the sets $\Upsilon_\epsilon$ with
$\epsilon\in {E}_{S'}$. Since $h_\epsilon\ge h_0(\delta')$ holds 
for all but finitely many $\epsilon\in {E}_{S'}$, estimate 
(\ref{eq:pos1}) holds for all but finitely many
$\epsilon\in {E}_{S'}$ as well, and it follows that
$$
\sum_{\epsilon\in {E}_{S'}} \mu(\Upsilon_\epsilon)<\infty.
$$
Hence, by the Borel--Cantelli lemma, $\mu(\Upsilon_{S'})=0$.
Let $\tilde \Upsilon_{S'}\subset G_{V_K}$ be the preimage of $\Upsilon_{S'}$.
Then $m(\tilde \Upsilon_{S'})=0$.
We denote by $\tilde \Upsilon_0$ the union of $\tilde \Upsilon_{S'}$ over finite subsets $S'$ of $S$.
Then $m(\tilde \Upsilon_0)=0$ as well.
Moreover, using Fubini's theorem and 
passing to a superset of $\tilde \Upsilon_0$, we can arrange that $\tilde \Upsilon_0$ additionally
satisfies the property that for every finite $S'\subset V_K$ and
$(f',f)\in G_{V_K\backslash S'}\times G_{S'}$ not belonging to $\tilde \Upsilon_0$,
the set $\{f'':\, (f'',f)\notin \tilde \Upsilon_0\}$ has full measure in $G_{V_K\backslash S'}$.
In particular,  it follows that if for $y\in G_{V_K\backslash S}$ and $x\in G_S$, we have
$(y,x)\notin \tilde \Upsilon_0$,
then there exists $(y',x')\notin \tilde \Upsilon_0$ such that
$x_v'=x_v$ for $v\in S'$, $x'_v\in {\rm G}(O_v)$ for $v\in V_K^f\cap (S\backslash S')$,
and $y'\in \mathcal{O}_{V_K\backslash S}(1)$. We shall use this property below.

Let
\begin{equation}\label{eq:om} 
\Omega'=\{x\in\Omega:\, \exists y\in \mathcal{O}_{V_K\backslash S}(1):\, (y,x^{-1}) \notin \tilde\Upsilon_0\}.
\end{equation}
Since
$$
(\mathcal{O}_{V_K\backslash S}(1)\times (\Omega\backslash \Omega')^{-1})\subset \tilde \Upsilon_0,
$$
it follows that the set $\Omega\backslash \Omega'$ has measure zero.

Let $G_{S}=\cup_{j\ge 1} \Omega_j$ be an exhaustion of $G_{S}$ by compact sets.
Then $Y:=\cup_{j\ge 1} \Omega_j'$ is a subset of $G_{S}$ of full measure.
Hence, it is suffices to show that given compact $\Omega\subset G_S$,
every $x\in\Omega'$ satisfies the claim of the theorem. 

For $x\in \Omega'$, there exists $\tilde \varsigma:=(y',(x')^{-1})\notin \tilde \Upsilon_0$ such that
$x_v'=x_v$ for $v\in S'$, $x'_v\in {\rm G}(O_v)$ for $v\in V_K^f\cap (S\backslash S')$,
and $y'\in \mathcal{O}_{V_K\backslash S}(1)$.

Since $\tilde \varsigma{\rm G}(K)\notin \Upsilon_0$, it follows that
$\tilde \varsigma{\rm G}(K)\notin \Upsilon_\epsilon$ for all but finitely many $\epsilon\in {E}_{S'}$.
That is, there exists $\epsilon_0(x,S',\delta)>0$ such that
for $\epsilon\in {E}_{S'}\cap (0,\epsilon_0(x,S',\delta))^{S'}$,
we have
$$
B_{h_\epsilon}^{-1}\tilde \varsigma{\rm G}(K) \cap \Phi_\epsilon\ne \emptyset,
$$
and
$$
(\mathcal{O}_{V_K\backslash S}(1)^{-1}B_{h_\epsilon})^{-1}(e,(x')^{-1}){\rm G}(K) \cap \Phi_\epsilon\ne \emptyset.
$$
Now we are in position to apply Proposition \ref{p:dual}.
It follows that there exists $z\in {\rm G}(K)$ such that
$$
\height(z)\ll_\Omega \max_{b\in \mathcal{O}_{V_K\backslash S}(1)^{-1}B_{h_\epsilon}} \height(b)\ll h_\epsilon
$$
and
\begin{align*}
&\|x'_v-z\|_v\le \epsilon_v\quad\hbox{for all $v\in S'$,}\\
&\|x'_v-z\|_v\le 1\quad\hbox{for all $v\in S\backslash S'$.}
\end{align*}
Since $x_v'\in {\rm G}(O_v)$ for $v\in V_K^f\cap (S\backslash S')$, it follows that
$z\in {\rm G}(O_{(V_K\backslash S)\cup S'})$
Hence, for every $\delta>0$ and $\epsilon\in {E}_{S'}\cap (0,\epsilon_0(x,S',\delta))^{S'}$, we have
$$
\omega_S(x,\epsilon)\ll_\Omega h_\epsilon=\left( \prod_{v\in S'} \epsilon_v^{-r_v \frac{\dim({\rm G})}{\mathfrak{a}_S({\rm G})}-\delta}
\right)^{\mathfrak{q}_{V_K\backslash S}(G)/2}.
$$
This implies that for all $\delta>0$ and 
sufficiently small $\epsilon$, depending on $\delta$ and $\Omega$, we also have  (absorbing the constant by passing from $\delta$ to $2\delta$)

$$
\omega_S(x,\epsilon)\le \left( \prod_{v\in S'} \epsilon_v^{-r_v \frac{\dim({\rm G})}{\mathfrak{a}_S({\rm G})}-2\delta}
\right)^{\mathfrak{q}_{V_K\backslash S}(G)/2}\,\,\,.
$$

Now in order to finish the proof of the theorem, we need to extend this estimate 
to the case when $\epsilon_v\in (0,\epsilon_0(x,S',\delta))$ for $v\in V_K^\infty\cap S'$.
For such $\epsilon_v$, there exists $\epsilon_v'\in {E}_v\cap
(0,\epsilon_0(x,S',\delta))$ such that $\epsilon_v'\le \epsilon_v\le 2\epsilon'_v$.
For $v\in V_K^f\cap S'$, we set $\epsilon_v'=\epsilon_v$.
Then
\begin{align}\label{eq:oooo}
\omega_S(x,\epsilon)&\le \omega_S(x,\epsilon')\le 
\left( \prod_{v\in S'} (\epsilon'_v)^{-r_v \frac{\dim({\rm G})}{\mathfrak{a}_S({\rm G})}-\delta}
\right)^{\mathfrak{q}_{V_K\backslash S}({\rm G})/2}\\
&\ll_{V_K^\infty\cap S'} \left( \prod_{v\in S'} \epsilon_v^{-r_v \frac{\dim({\rm G})}{\mathfrak{a}_S({\rm G})}-\delta}
\right)^{\mathfrak{q}_{V_K\backslash S}({\rm G})/2}\nonumber 
\end{align}
for every $\delta>0$ and $\epsilon=(\epsilon_v)_{v\in S'}$ with
$\epsilon_v\in I_v\cap (0,\epsilon_0(x,S',\delta))$.
This also implies that
$$
\omega_S(x,\epsilon)
\le \left( \prod_{v\in S'} \epsilon_v^{-r_v \frac{\dim({\rm G})}{\mathfrak{a}_S({\rm G})}-2\delta}
\right)^{\mathfrak{q}_{V_K\backslash S}({\rm G})/2}
$$
for all $\delta>0$ and sufficiently small $\epsilon$ (depending on $\delta$ and $V_K^\infty\cap S'$).
This completes the proof.
\end{proof}

\begin{proof}[Proof of Theorem \ref{th:g2}]
The proof of the theorem follows the same strategy as the proof of Theorem \ref{th:g1},
but we choose the parameter $h_\epsilon$ to be
$$
h_\epsilon=\left( \prod_{v\in S'} \epsilon_v^{-\frac{r_v\dim({\rm G})+\sigma_S}{\mathfrak{a}_S({\rm G})}-\delta}
\right)^{\mathfrak{q}_{V_K\backslash S}({\rm G})/2}
$$
with $\delta>0$.

We fix a bounded subset $\Omega$ of $G_S$ and consider the sets
$$
\Upsilon_\epsilon=\{\varsigma\in\Upsilon:\, B_{h_\epsilon}^{-1}\varsigma\cap \Phi_{\epsilon}=\emptyset\},
$$
indexed by $\epsilon=(\epsilon_v)_{v\in S'}$ with finite $S'\subset S$ and $\epsilon_v\in I_v$, where
the sets $\Phi_\epsilon$ is defined in Proposition \ref{p:dual}.
Arguing as in the proof of Theorem \ref{th:g1}, we obtain the estimate
\begin{equation}\label{eq:uuuu}
\mu(\Upsilon_\epsilon) \ll_{\Omega,\delta'} \prod_{v\in S'} \epsilon_v^{\theta_v},
\end{equation}
where
\begin{align*}
\theta_v &=\left(-\frac{r_v\dim({\rm G})+\sigma_S}{\mathfrak{a}_S({\rm G})}-\delta\right)\frac{\mathfrak{q}_{V_K\backslash S}({\rm
    G})}{2}\left(-\frac{2\mathfrak{a}_S({\rm G})}{\mathfrak{q}_{V_K\backslash S}({\rm G})}+2\delta'\right)-r_v\dim({\rm G})\\
&= \sigma_S+\mathfrak{a}_S({\rm G}) \delta-\delta' \cdot \mathfrak{q}_{V_K\backslash S}({\rm G})\left(\frac{r_v\dim({\rm G})+\sigma_S}{\mathfrak{a}_S({\rm G})}+\delta\right).
\end{align*}
This estimate is valid for every $\delta,\delta'>0$ and $\epsilon$ as above
provided that $h_\epsilon\ge h_0(\delta')$. Moreover, we emphasize that this
estimate is uniform over finite $S'\subset S$.
We choose $\delta'=\delta'(\delta)>0$ such that $\theta_v\ge \theta>\sigma_S$. 

Let
$$
{E}_{S}=\{(\epsilon_v)_{v\in S}:\epsilon_v=q_v^{-n_v}
\hbox{ with $n_v\ge 0$, $n_v=0$ for a.e. $v$}\},
$$
where we set $q_v=2$ for $v\in V_K^\infty$.
We note that $h_\epsilon\ge h_0(\delta')$ for all but finitely many $\epsilon\in E_S$.
Therefore, estimate (\ref{eq:uuuu}) holds for all but finitely many $\epsilon\in E_S$,
and the sum $\sum_{\epsilon\in {E}_{S}} \mu(\Upsilon_\epsilon)$
can be estimated (except possibly finitely many terms) by 
$$
\sum_{\epsilon\in E_S} \left(\prod_{v\in S} \epsilon_v^{\theta_v}\right)
\le  \prod_{v\in S} (1-q_v^{-\theta})^{-1},
$$
where the last product converges because $\theta>\sigma_S$.

Let $\Upsilon_0$ be the $\limsup$ of the sets $\Upsilon_\epsilon$ as $\epsilon\in {E}_{S}$.
Then by the Borel-Cantelli lemma, $\mu(\Upsilon_0)=0$. 
Let $\tilde \Upsilon_0\subset G_{V_K}$ be the preimage of $\Upsilon_0$.
Then $m(\tilde \Upsilon_0)=0$. Moreover, as in the proof of Theorem \ref{th:g1}(i),
we can pass to a superset of $\tilde \Upsilon_0$ to arrange that 
if for $y\in G_{V_K\backslash S}$ and $x\in G_S$, we have
$(y,x)\notin \tilde \Upsilon_0$,
then there exists $(y',x')\notin \tilde \Upsilon_0$ such that
$x_v'=x_v$ for $v\in S'$, $x'_v\in {\rm G}(O_v)$ for $v\in V_K^f\cap (S\backslash S')$,
and $y'\in \mathcal{O}_{V_K\backslash S}(1)$. 

Given bounded $\Omega\subset G_S$, we define $\Omega'$ as in (\ref{eq:om}).
Then $\Omega\backslash \Omega'$ has measure zero, and it suffices to show that every $x\in\Omega'$
satisfies the claim of the theorem.

For $x\in \Omega'$, there exists $\tilde \varsigma:=(y',(x')^{-1})\notin \tilde \Upsilon_0$ such that
$x_v'=x_v$ for $v\in S'$, $x'_v\in {\rm G}(O_v)$ for $v\in V_K^f\cap (S\backslash S')$,
and $y'\in \mathcal{O}_{V_K\backslash S}(1)$.
Since $\tilde \varsigma{\rm G}(K)\notin \Upsilon_0$,  we conclude that $\tilde \varsigma{\rm G}(K)\notin \Upsilon_\epsilon$
for all but finitely many $\epsilon\in {E}_S$.
In particular, there exists $\epsilon_v^0(x,\delta)\in (0,1]$, $v\in S$, such that 
$\epsilon_v^0(x,\delta)=1$ for almost all $v$ and 
for all 
$\epsilon=(\epsilon_v)_{v\in S'}$ with finite $S'\subset S$
and $\epsilon_v\in {E}_v\cap  (0,\epsilon_v^0(x,\delta))$,
we have
$$
B_{h_\epsilon}^{-1}\tilde \varsigma{\rm G}(K)\cap \Phi_\epsilon\ne \emptyset.
$$
Applying Proposition \ref{p:dual}, we deduce that
there exists $z\in {\rm G}(K)$ such that
$$
\height(z)\ll_\Omega h_\epsilon
$$
and 
\begin{align*}
&\|x'_v-z\|_v\le \epsilon_v\quad\hbox{for all $v\in S'$,}\\
&\|x'_v-z\|_v\le 1\quad\hbox{for all $v\in S\backslash S'$,}
\end{align*}
provided that $\epsilon_v\in \{q_v^{-n}\}_{n\ge 0}\cap  (0,\epsilon_v^0(x,\delta))$.
Since $x_v'\in {\rm G}(O_v)$ for $v\in V_K^f\cap (S\backslash S')$, it follows that
$z\in {\rm G}(O_{(V_K\backslash S)\cup S'})$.

We conclude that
\begin{equation}\label{eq:last}
\omega_S(x,\epsilon)\ll_\Omega h_\epsilon=\left( \prod_{v\in S'} \epsilon_v^{-\frac{r_v\dim(G)+\sigma_S}{\mathfrak{a}_S(G)}-\delta}
\right)^{\mathfrak{q}_{V_K\backslash S}(G)/2}
\end{equation}
for every finite $S'\subset S$ and  $\epsilon\in {E}_{S'}$ such that 
$\epsilon_v\in \{q_v^{-n}\}_{n\ge 0}\cap  (0,\epsilon_v^0(x,\delta))$.
Moreover, this implies that for all $\delta>0$ and sufficiently small $\epsilon$
(depending on $\delta$ and $\Omega$), we have
$$
\omega_S(x,\epsilon)\le\left( \prod_{v\in S'} \epsilon_v^{-\frac{r_v\dim(G)+\sigma_S}{\mathfrak{a}_S(G)}-2\delta}
\right)^{\mathfrak{q}_{V_K\backslash S}(G)/2}
$$
Finally, it remains to extend this estimate to $\epsilon_v\in (0,\epsilon_v^0(x,\delta))$ when
$v\in V_K^\infty\cap S'$. This can be done as in (\ref{eq:oooo}).
\end{proof}

\subsection{Diophantine approximation at every point on the group variety}

\begin{proof}[Proof of Theorem \ref{th:g1}(ii)]
Without loss of generality, we may assume that $\Omega=\prod_{v\in S} \Omega_v$
where $\Omega_v$ is a bounded subset $G_v$ and $\Omega_v\supset {\rm G}(O_v)$ for all
$v\in V_K^f$. 

Let the sets $\tilde \Psi_\epsilon\subset G_{V_K}$ be as defined in Proposition \ref{p:dual2}
indexed by $\epsilon=(\epsilon_v)_{v\in S'}$ with finite $S'\subset S$ and $\epsilon_v\in I_v$.
We set 
$$
h_\epsilon=\left( \prod_{v\in S'} \epsilon_v^{-r_v \frac{\dim({\rm G})}{\mathfrak{a}_S({\rm G})}-\delta} \right)^{\mathfrak{q}_{V_K\backslash S}({\rm G})}
$$
with $\delta>0$ and
$$
\Upsilon_\epsilon=\{\varsigma\in\Upsilon:\, B_{h_\epsilon}^{-1}\varsigma\cap \tilde \Psi_{\epsilon} {\rm G}(K)=\emptyset\}.
$$
Let $\psi_\epsilon$ be the characteristic function of the set $\Psi_{\epsilon}:=\tilde \Psi_{\epsilon}
{\rm G}(K)\subset\Upsilon$. As in the proof of Theorem \ref{th:g3}, we get
$$
\left\| \pi_{V_K\backslash S}(\beta_h)\psi_\epsilon -\mu(\Psi_\epsilon)\right\|_2
\ll_{\delta'} h^{-\frac{\mathfrak{a}_S({\rm G})}{\mathfrak{q}_{V_K\backslash S}({\rm G})}+\delta'}\mu(\Psi_\epsilon)^{1/2}
$$
for every $\delta'>0$ and  $h\ge  h_0(\delta')$.
Since $\pi_{V_K\backslash S}(\beta_{h_\epsilon})\psi_\epsilon=0$ on $\Upsilon_\epsilon$,
it follows that 
$$
\mu(\Upsilon_\epsilon)\ll_{\delta'} 
h_\epsilon^{-\frac{2\mathfrak{a}_S({\rm G})}{\mathfrak{q}_{V_K\backslash S}({\rm G})}+2\delta'}\mu(\Psi_\epsilon)^{-1},
$$
and using (\ref{eq:llow2}), we conclude that
\begin{equation}\label{eq:est1}
\mu(\Upsilon_\epsilon)
\ll_{\Omega,\delta'} \prod_{v\in S'} \epsilon_v^{\theta_v},
\end{equation}
 where
\begin{align*}
\theta_v&=\left(r_v \frac{\dim({\rm G})}{\mathfrak{a}_S({\rm G})}+\delta\right)\mathfrak{q}_{V_K\backslash S}({\rm
  G})\left(\frac{2\mathfrak{a}_S({\rm G})}{\mathfrak{q}_{V_K\backslash S}({\rm G})}-2\delta'\right)-
r_v\dim ({\rm G})\\
&=r_v\dim({\rm G}) +\delta\cdot \mathfrak{q}_{V_K\backslash S}({\rm
  G}) \left(\frac{2\mathfrak{a}_S({\rm G})}{\mathfrak{q}_{V_K\backslash S}({\rm
      G})}-2\delta'\right)-\delta'\cdot 2r_v \frac{\dim({\rm G})}{\mathfrak{a}_S({\rm G})}.
\end{align*}
This estimate holds provided that $h_\epsilon\ge h_0(\delta')$, so that
it holds for all but finitely many $\epsilon\in E_S$.

We choose $\delta'=\delta'(\delta)$ such that $\theta_v>r_v\dim({\rm G})$.

For $x\in \Omega$, there exists $x'\in \Omega$ such that $x_v'=x_v$ for $v\in S'$ and $x_v'\in {\rm G}(O_v)$
for $v\in V_K^f\cap (S\backslash S')$.
For $\varsigma:=(e,(x')^{-1}){\rm G}(K)$, we have by (\ref{eq:llow22})
\begin{equation}\label{eq:est2}
\mu(\tilde \Psi_\epsilon^{-1} \varsigma)  \gg_\Omega  \prod_{v\in S'} \epsilon_v^{r_v \dim({\rm G})}.
\end{equation}
Comparing (\ref{eq:est1}) and (\ref{eq:est2}), we conclude that
the inequality
\begin{equation}\label{eq:ennn}
\mu(\tilde \Psi_\epsilon^{-1} \varsigma)> \mu(\Upsilon_\epsilon)
\end{equation}
holds for all but finitely many $\epsilon\in E_S$.
If (\ref{eq:ennn}) holds, then $\tilde \Psi^{-1}_\epsilon \varsigma\nsubseteq \Upsilon_\epsilon$,
and 
$$
B_{h_\epsilon}^{-1} \tilde \Psi^{-1}_\epsilon \varsigma\cap \tilde \Psi_\epsilon {\rm G}(K)\ne \emptyset.
$$
Therefore, by Proposition \ref{p:dual2}, for all but finitely many $\epsilon\in E_S$,
there exists $z\in {\rm G}(K)$ such that
\begin{equation*}
\height(z)\ll_\Omega h_\epsilon
\end{equation*}
and
\begin{align*}
&\|x'_v-z\|_v\le \epsilon_v\quad\hbox{for all $v\in S'$,}\\
&\|x'_v-z\|_v\le 1\quad\hbox{for all $v\in S\backslash S'$.}
\end{align*}
Since $x_v'\in {\rm G}(O_v)$ for $v\in V_K^f\cap (S\backslash S')$,
it follows that $z\in {\rm G}(O_{(V_K\backslash S)\cup S'})$.

This shows that there exists $\epsilon_v^0(\Omega,\delta)\in (0,1]$ satisfying
$\epsilon_v^0(\Omega,\delta)=1$ for almost all $v$ such that
for every $x\in \Omega$ and $\epsilon\in E_S$ with $\epsilon_v\in (0,\epsilon_v^0(\Omega,\delta))$,
we have the estimate
$$
\omega_S\left(x,(\epsilon_v)_{v\in S'}\right)\ll_\Omega h_\epsilon=\left( \prod_{v\in S'} \epsilon_v^{-r_v \frac{\dim({\rm G})}{\mathfrak{a}_S({\rm G})}-\delta}
\right)^{\mathfrak{q}_{V_K\backslash S}({\rm G})}.
$$
Therefore, for all $\delta>0$ and sufficiently small $\epsilon$ (depending on $\delta$ and $\Omega$),
we also have
$$
\omega_S\left(x,(\epsilon_v)_{v\in S'}\right)\le\left( \prod_{v\in S'} \epsilon_v^{-r_v \frac{\dim({\rm G})}{\mathfrak{a}_S({\rm G})}-2\delta}
\right)^{\mathfrak{q}_{V_K\backslash S}({\rm G})}.
$$
Now it remains to extend this estimate to the case when $\epsilon_v\in (0,\epsilon_v^0(\Omega,\delta))$
for $v\in V_K^\infty\cap S'$,
which can be achieved as in (\ref{eq:oooo}).
This completes the proof of the theorem.
\end{proof}

\subsection{Diophantine approximation on homogeneous varieties}

Let ${\rm X}$ be an quasi-affine algebraic variety defined over $K$ equipped with 
a transitive action of a  connected almost simple algebraic $K$-group ${\rm G}$.
Let $S$ be a subset of $V_K$ and $S'$ a finite subset of $S$.
Before we start the proof of Theorem \ref{th:g3}, we need to describe the
structure of the topological closure $\overline{{\rm X}(O_{(V_K\backslash S)\cup  S'})}$
in $X_{S'}$.

\begin{Lemma}\label{l:closure}
Assume that ${\rm G}$ is isotropic over $V_K\backslash S$, and let
$p: \tilde {\rm G}\to {\rm G}$ denote the simply connected cover of ${\rm G}$.
Then the closure $\overline{{\rm X}(O_{(V_K\backslash S)\cup S'})}$ in $X_{S'}$ is open and
$$
\overline{{\rm X}(O_{(V_K\backslash S)\cup
    S'})}=p(\tilde G_{S'}){\rm X}(O_{(V_K\backslash S)\cup
    S'}).
$$
Moreover, it is a union of finitely many orbits of $p(\tilde G_{S'})$.
\end{Lemma}

\begin{proof}
Given $x^0\in {\rm X}(O_{(V_K\backslash S)\cup S'})$, we consider the map
$$
P:\tilde {\rm G}\to {\rm X}: g\mapsto p(g)x^0.
$$
Let 
$$
\Gamma=\tilde {\rm G}(K)\cap \left(\bigcap_{v\in (S\backslash S')\cap V_K^f} p^{-1}({\rm G}(O_v))\right).
$$
Then $P(\Gamma)\subset {\rm X}(O_{(V_K\backslash S)\cup S'})$. Since $\tilde G$ is isotropic
over $V_K\backslash S$, the subgroup $\Gamma$ is dense in $\tilde G_{S'}$, and
$$
P(\tilde G_{S'})\subset \overline{P(\Gamma)}\subset \overline{{\rm X}(O_{(V_K\backslash S)\cup
    S'})}.
$$
Since the variety ${\rm X}$ is a homogeneous space of $\tilde {\rm G}$ with the action
$g\cdot x=p(g)x$, $x\in X$, it follows from \cite[\S3.1]{PlaRa} that every orbit of $\tilde G_{S'}$ is open and
closed in $X_{S'}$. This implies the first claim. 

The last claim follows from 
finiteness of Galois cohomology over local fields (see \cite[\S6.4]{PlaRa}).
\end{proof}

\begin{proof}[Proof of Theorem \ref{th:g3}(i)]
According to Lemma \ref{l:closure}, the set $\overline{{\rm X}(O_{(V_K\backslash S)\cup S'})}$ is 
a finite union of open (and closed) orbits of $p(\tilde G_{S'})$.
We intend to show that given $x^0\in {\rm X}(O_{(V_K\backslash S)\cup S'})$ there exists a set $Y$
of full measure in  $p(\tilde G_{S'})x^0$ whose points satisfies the claim of the theorem.
In fact, we prove that for every bounded $Y\subset p(\tilde G_{S'})x^0$
there exists a conull subset $Y'\subset Y$ whose points satisfy the claim
of the theorem. This will complete the proof.

Let $\Omega$ be a bounded subset of $p(\tilde G_{S'})$ such that
$\Omega x^0=Y$.

We set 
$$
U^0=\prod_{v\in V_K^f\cap (S\backslash S')} (U_v\cap {\rm G}(O_v))\quad\hbox{and}\quad U=U_{V_K^f\backslash S}U^0.
$$
Note that that $U_v={\rm G}(O_v)$ for almost all $v$.
Since both $U_v$ and ${\rm G}(O_v)$ are open and compact, it follows
that the subgroup $U_v\cap {\rm G}(O_v)$ has finite index in $U_v$.
Hence, $U^0$ has finite index in $U_{V_K^f\cap (S\backslash S')}$.

Recall that $G^U$ denotes the kernel of $U$-invariant automorphic characters of $G_{V_K}$
(see Section \ref{sec:mean}). We note that ${\rm G}(K)\subset G^U$, and
$G^U$ is a finite index subgroup of $G_{V_K}$ by Lemma \ref{l:characters}.
Let $\beta'_h$ be the Haar-uniform probability measure supported on $B_h':=U^0(B_h\cap G^U)$.
By Theorem \ref{c:mean}, for every $\phi\in L^2(\Upsilon)$ such that
$\supp(\phi)\subset G^U/{\rm G}(K)$ and $h,\delta'>0$,
$$
\left\|\pi_{(V_K\backslash S)\cup (V_K^f\backslash S')}(\beta'_h)\phi 
-\left(\int_\Upsilon \phi\, d\mu\right)\xi_U\right\|_2\ll_{\delta'}
m_{V_K\backslash S}(B_h\cap G^U)^{-\frac{1}{\mathfrak{q}_{V_K\backslash S}({\rm G})}+\delta'}\|\phi\|_2.
$$
By Lemmas \ref{l:ass} and \ref{l:finite}, 
$$
m_{V_K\backslash S}(B_h\cap G^U)\gg m_{V_K\backslash S}(B_h)\gg_\Omega h^{\mathfrak{a}_S({\rm G})-\delta'}
$$
for every $\delta'>0$ and $h\ge h_0(\delta')$. Hence, combining these estimates,
we conclude that for every $\delta'>0$ and $h\ge h_0(\delta')$, we have
\begin{equation}\label{eq:norm1}
\left\|\pi_{(V_K\backslash S)\cup (V_K^f\backslash S')}(\beta'_h)\phi 
-\left(\int_\Upsilon \phi\, d\mu\right)\xi_U\right\|_2\ll_{\delta'}
h^{-\frac{\mathfrak{a}_S({\rm G})}{\mathfrak{q}_{V_K\backslash S}({\rm G})}+\delta'}\|\phi\|_2.
\end{equation}
We apply (\ref{eq:norm1}) in the case when $\phi$ is the characteristic function $\phi_\epsilon$ of
the set $\Phi_\epsilon$ introduced in Proposition \ref{p:dual1_prime},
which is indexed by $\epsilon=(\epsilon_v)_{v\in S'}$ with 
finite $S'\subset S$ and $\epsilon_v\in I_v\cap (0,\epsilon_0)$.
Note that it follows from the construction of the sets $\Phi_\epsilon$ that
we can arrange that for sufficiently small $\epsilon$, we have $\Phi_\epsilon\subset G^U/{\rm G}(K)$,
and hence (\ref{eq:norm1}) is applicable to $\phi=\phi_\epsilon$. We obtain
\begin{equation}\label{eq:norm11}
\left\|\pi_{(V_K\backslash S)\cup (V_K^f\backslash S')}(\beta'_h)\phi_\epsilon 
-\mu(\Phi_\epsilon)\xi_U\right\|_2\ll_{\delta' }
h^{-\frac{\mathfrak{a}_S({\rm G})}{\mathfrak{q}_{V_K\backslash S}({\rm G})}+\delta'}\mu(\Phi_\epsilon)^{1/2}
\end{equation}
for every $\delta'>0$ and $h\ge h_0(\delta')$.
Let 
\begin{equation}\label{eq:heeee}
h_\epsilon=\left(\prod_{v\in S'} \epsilon_v^{-r_v\frac{\dim({\rm X})}{\mathfrak{a}_S({\rm
        G})}-\delta}\right)^{\mathfrak{q}_{V_K\backslash S}({\rm G})/2}
\end{equation}
with $\delta>0$, and 
$$
\Upsilon_\epsilon=\{\varsigma\in G^U/{\rm G}(K):\, (B_{h_\epsilon}')^{-1}\varsigma\cap \Phi_\epsilon=\emptyset\}.
$$
For $\varsigma\in \Upsilon_\epsilon$, we have
$$
\pi_{(V_K\backslash S)\cup (V_K^f\backslash S')}(\beta'_{h_\epsilon})\phi_\epsilon(\varsigma)=0.
$$
Hence, (\ref{eq:norm11}) implies the estimate
\begin{align*}
\mu(\Upsilon_\epsilon) &\ll_{\delta'} |G_{V_K}:G^U|^{-1} h_\epsilon^{-\frac{2 \mathfrak{a}_S({\rm G})}{\mathfrak{q}_{V_K\backslash S}({\rm
    G})}+2\delta'} \mu(\Phi_\epsilon)^{-1},
\end{align*}
and using (\ref{eq:heeee}) and (\ref{eq:low1}), we conclude that
\begin{align}\label{eq:u_eee}
\mu(\Upsilon_\epsilon)
&\ll_{\Omega,\delta'} \prod_{v\in S'} \epsilon_v^{\theta_v},
\end{align}
where 
\begin{align*}
\theta_v =\mathfrak{a}_S({\rm G}) \delta-\delta'\cdot \mathfrak{q}_{V_K\backslash S}({\rm G})\left(r_v \frac{\dim({\rm X})}{\mathfrak{a}_S({\rm G})}+\delta\right).
\end{align*}
This estimate holds provided that $h_\epsilon\ge h_0(\delta')$.
Hence, it holds for all but finitely many $\epsilon\in E_{S'}\cap (0,\epsilon_0)^{S'}$.

We choose $\delta'=\delta'(\delta)>0$ sufficiently small, so that $\theta_v>0$.

Now we may argue as in the proof of Theorem \ref{th:g1}(i).
We denote by $\Upsilon_0$ the $\limsup$ of the sets $\Upsilon_\epsilon$
with $\epsilon\in {E}_{S'}\cap (0,\epsilon_0)^{S'}$. Since (\ref{eq:u_eee}) holds for all but finitely many
$\epsilon\in E_{S'}$, it follows that
$$
\sum_{\epsilon\in E_{S'}\cap (0,\epsilon_0)^{S'}} \mu(\Upsilon_\epsilon)<\infty,
$$
and by the Borel--Cantelli lemma, $\mu(\Upsilon_0)=0$.
We fix a bounded neighbourhood $V$ of identity in $G_{V_K^\infty\backslash S'}$ contained in $G^U$.
Then $UV$ is a neighbourhood of identity in $G_{V_K\backslash S'}$.
If we set
$$
\Omega'=\{g\in \Omega:\,\, \exists y\in UV:\, (y,g^{-1}){\rm G}(K)\notin\Upsilon_0\},
$$
then 
$$
(UV\times (\Omega\backslash \Omega')^{-1}){\rm G}(K)\subset \Upsilon_0,
$$
and, hence, the set $\Omega\backslash \Omega'$ has measure zero.
This implies that the subset $Y':=\Omega' x^0$
has full measure in $Y=\Omega x^0$.

For $g\in \Omega'$, there exists $y\in UV$ such that 
$\varsigma:= (y, g^{-1}){\rm G}(K)\notin \Upsilon_0$. This implies that there 
exists $\epsilon_0(g,\delta)>0$ such that for 
$\epsilon\in E_{S'}\cap (0,\epsilon_0(g,\delta))^{S'}$, we have
$$
(B_{h_\epsilon}')^{-1}\varsigma\cap \Phi_\epsilon\ne \emptyset, 
$$
and hence,
$$
((UV)^{-1}B_{h_\epsilon}')^{-1}(e,x^{-1})\cap \Phi_\epsilon\ne \emptyset.
$$
Therefore, it follows from Proposition \ref{p:dual1_prime} that
there exists $z\in {\rm G}(O_{(V_K\backslash S)\cup S'})$ such that
$$
\height(z)\ll_\Omega \max_{b\in (UV)^{-1}B_{h_\epsilon}'} \height(b)\ll h_\epsilon,
$$
and for $x=gx^0\in Y'$,
$$
\|x_v-z x^0\|_v \le \epsilon_v\quad \hbox{for all $v\in S'$.}
$$
We have  $z x^0 \in {\rm X}(O_{(V_K\backslash S)\cup S'})$ and $\height(z x^0)\ll \height(z)$.

Since these estimates are valid for arbitrary bounded $Y$, we conclude that 
for every $\delta>0$, almost every $x\in \overline{{\rm X}(O_{(V_K\backslash S)\cup S'})}$, and
$(\epsilon_v)_{v\in S'}$ with finite $S'\subset S$ and
$\epsilon_v\in E_v\cap (0,\epsilon_0(x,S',\delta))$, we have
$$
\omega_S(x,(\epsilon_v)_{v\in S'})\ll_\Omega h_\epsilon= \left(
\prod_{v\in S'} \epsilon_v^{-r_v \frac{\dim({\rm X})}{\mathfrak{a}_S({\rm G})}-\delta}\right)^{\mathfrak{q}_{V_K\backslash S}({\rm G})/2}.
$$
Moreover, this implies that 
$$
\omega_S(x,(\epsilon_v)_{v\in S'})\le\left(
\prod_{v\in S'} \epsilon_v^{-r_v \frac{\dim({\rm X})}{\mathfrak{a}_S({\rm G})}-2\delta}\right)^{\mathfrak{q}_{V_K\backslash S}({\rm G})/2}
$$
for every $\delta>0$ and sufficiently small $\epsilon$ (depending on $\delta$ and $\Omega$).
Finally, it remains to extend this estimate to $\epsilon_v\in (0,\epsilon_0(x,S',\delta))$
for $v\in V_K^\infty\cap S'$, which can be done as in (\ref{eq:oooo}).
\end{proof}

\begin{proof}[Proof of Theorem \ref{th:g3}(ii)]
By Lemma \ref{l:closure}, the set $\overline{{\rm X}(O_{(V_K\backslash S)\cup S'})}$ is 
a finite union of open (and closed) orbits of $p(\tilde G_{S'})$. Hence, it suffices to prove
the theorem for every  compact subset $Y$ contained in  $p(\tilde G_{S'})x^0$
with $x^0\in {\rm X}(O_{(V_K\backslash S)\cup S'})$.

Let $\Omega$ be a bounded subset of $p(\tilde G_{S'})$ such that
$\Omega x^0=Y$. 
Take $g\in \Omega$ and set $\varsigma:=(e,g^{-1}){\rm G}(K)\in \Upsilon$.
We note that $p(\tilde G_{S'})\subset G^U$. In particular, $g\in G^U$.

Let $B'_h$ and $U$ be defined as in the proof of Theorem \ref{th:g3}(i),
and let $\tilde \Psi_\epsilon\subset G_{V_K}$ be the sets defined as in Proposition \ref{p:dual2_prime}
which are indexed by $\epsilon=(\epsilon_v)_{v\in S'}$ with finite $S'\subset S$ and
$\epsilon_v\in I_v\cap (0,\epsilon_0)$.
As in the proof of Theorem \ref{th:g3}(i), we will work on the space $G^U/{\rm G}(K)$.
We note that taking the components of $\epsilon$, $L_v^0$, and $\mathcal{O}_v$
sufficiently small, we can arrange that
$\tilde \Psi_\epsilon \subset G^U$. 
Then Theorem \ref{c:mean} is applicable to
the function $\psi_\epsilon$ which is  the characteristic function of the set
$\Psi_\epsilon:=\tilde \Psi_\epsilon {\rm G}(K)\subset G^U/{\rm G}(K)$. Hence, as in the proof
of Theorem \ref{th:g3}(i),  we have
\begin{equation}\label{eq:norm111}
\left\|\pi_{(V_K\backslash S)\cup (V_K^f\backslash S')}(\beta'_h)\psi_\epsilon 
-\mu(\Psi_\epsilon)\xi_U\right\|_2\ll_{\delta'}
h^{-\frac{\mathfrak{a}_S({\rm G})}{\mathfrak{q}_{V_K\backslash S}({\rm G})}+\delta'}\mu(\Psi_\epsilon)^{1/2}
\end{equation}
for every $\delta'>0$ and $h\ge h_0(\delta')$.
We set 
$$
h_\epsilon=\left( \prod_{v\in S'} \epsilon_v^{-r_v \frac{\dim({\rm G})}{\mathfrak{a}_S({\rm G})}-\delta} \right)^{\mathfrak{q}_{V_K\backslash S}({\rm G})}
$$
with $\delta>0$, and
$$
\Upsilon_\epsilon=\{\varsigma\in G^U/{\rm G}(K):\, (B'_{h_\epsilon})^{-1}\varsigma\cap \tilde \Psi_{\epsilon} {\rm G}(K)=\emptyset\}.
$$
Since $\pi_{(V_K\backslash S)\cup (V_K^f\backslash S')}(\beta'_{h_\epsilon})\psi_\epsilon =0$ on $\Upsilon_\epsilon$,
inequality (\ref{eq:norm111}) implies that
$$
\mu(\Upsilon_\epsilon)\ll_{\delta'} |G_{V_K}:G^U|^{-1} h_\epsilon^{-\frac{2\mathfrak{a}_S({\rm G})}{\mathfrak{q}_{V_K\backslash S}({\rm G})}+2\delta'}\mu(\Psi_\epsilon)^{-1}
$$
when $h_\epsilon\ge h_0(\delta')$.
Combining this estimate with (\ref{eq:low1_1}), we conclude that 
\begin{equation}\label{eq:abs_last}
\mu(\Upsilon_\epsilon)\ll_{\Omega,\delta'} \prod_{v\in S'} \epsilon_v^{\theta_v},
\end{equation}
where 
\begin{align*}
\theta_v&=r_v\dim({\rm X}) +\delta\cdot \mathfrak{q}_{V_K\backslash S}({\rm
  G}) \left(\frac{2\mathfrak{a}_S({\rm G})}{\mathfrak{q}_{V_K\backslash S}({\rm
      G})}-2\delta'\right)-\delta'\cdot 2r_v \frac{\dim({\rm G})}{\mathfrak{a}_S({\rm X})}.
\end{align*}
This estimate holds for all $\epsilon$ satisfying $h_\epsilon\ge h_0(\delta')$,
and hence for all $\epsilon$
with sufficiently small components (depending on $\delta'$).

We pick $\delta'=\delta'(\delta)>0$ such that $\theta_v>r_v\dim({\rm G})$.

Let $g\in \Omega$ and $\varsigma:=(e,g^{-1}){\rm G}(K)$.
Then by (\ref{eq:low1_12}),
$$
\mu(\tilde \Psi^{-1}_\epsilon \varsigma)\gg_\Omega\prod_{v\in S'} \epsilon_v^{r_v\dim({\rm X})},
$$
Comparing this estimate with (\ref{eq:abs_last}), we conclude that  
$$
\mu(\tilde \Psi^{-1}_\epsilon \varsigma)>\mu(\Upsilon_\epsilon)
$$
provided that $\epsilon_v\le \epsilon_0(Y,S',\delta)$, $v\in S'$.
Then $\tilde \Psi^{-1}_\epsilon \varsigma\nsubseteq\Upsilon_\epsilon$, and
$$
(B_h')^{-1} \tilde \Psi^{-1}_\epsilon \varsigma\cap \Psi_\epsilon\ne \emptyset.
$$
Hence, Proposition \ref{p:dual2_prime} implies that
there exists $z\in {\rm G}(O_{(V_K\backslash S)\cup S'})$ such that
$$
\height(z)\ll_\Omega \max_{b\in B_{h_\epsilon}'} \height(b)\ll h_\epsilon=\left( \prod_{v\in S'} \epsilon_v^{-r_v \frac{\dim({\rm G})}{\mathfrak{a}_S({\rm G})}-\delta} \right)^{\mathfrak{q}_{V_K\backslash S}({\rm G})},
$$
and for $x=gx_0$, we have
$$
\|x_v-zx^0\|\le \epsilon_v\quad\hbox{for all $v\in S'$.}
$$
Since $\height(z x^0)\ll \height(z)$, we deduce that for every $\delta>0$ and 
$(\epsilon_v)_{v\in S'}$ with  $\epsilon_v\in I_v\cap (0,\epsilon_0(Y,\delta))$, we have
$$
\omega_S(x, (\epsilon_v)_{v\in S'})\ll_\Omega \left( \prod_{v\in S'} \epsilon_v^{-r_v \frac{\dim({\rm
        G})}{\mathfrak{a}_S({\rm G})}-\delta} \right)^{\mathfrak{q}_{V_K\backslash S}({\rm G})}.
$$
Moreover, this implies that for all sufficiently small $\epsilon$ (depending on $\delta$ and $\Omega$),
we have
$$
\omega_S(x, (\epsilon_v)_{v\in S'})\ll_\Omega \left( \prod_{v\in S'} \epsilon_v^{-r_v \frac{\dim({\rm
        G})}{\mathfrak{a}_S({\rm G})}-2\delta} \right)^{\mathfrak{q}_{V_K\backslash S}({\rm G})}.
$$
This completes the proof of the theorem.
\end{proof}

\section*{Notation}

{\small

\begin{tabular}{ll}
$K$ & a number field,\\
$V_K$ & the set of places of $K$,\\
$V_K^f$ & the subset of non-Archimedean places,\\
$V_K^\infty$ & the subset of Archimedean places,\\
$O$ & the ring of integers of $K$,\\
$O_S$ & the ring of $S$-integers of $K$,\\
$K_v$ & the completion of $K$ at $v\in V_K$,\\
$O_v$ & the ring of integers of $K_v$,\\
$q_v$ & the cardinality of the residue field of $O_v$,\\
$r_v$ & $2$ if $v$ is a complex place and $1$ otherwise,\\
$I_v$ & $(0,1)$ for $v\in V_K^\infty$ and $\{q_v^{-n}\}_{n\ge 1}$ for $v\in V_K^f$,\\
$G_v$ & the set of $K_v$-points of ${\rm G}$,\\
$U_v$ & the maximal compact subgroup of $G_v$ chosen in Section \ref{sec:semisimple_and_spherical},\\
$G_S$ & the restricted direct product of $(G_v,U_v)$ with $v\in S$,\\
$m_{G_v}$ & the Haar measure on $G_v$ normalised,\\
&so that $m_{G_v}(U_v)=1$ for non-Archimedean $v$,\\
$m_S$ & the product of $m_{G_v}$ over $v\in S$,\\
$m_v$ & the Haar measure on $G_v$ defined by ${\rm G}$-invariant differential form,\\
$m$ & the Tamagawa measure on $G_{V_K}$ which is the product of $m_v$'s,\\
$\height$ & the height function, see (\ref{eq:height}),~(\ref{eq:hhh}),\\
$\omega_S$ & the Diophantine approximation function, see (\ref{eq:omega_S}),\\
$\mathfrak{a}_S({\rm X})$ & the exponent of a variety ${\rm X}$, see (\ref{eq:ass_omega}),\\
$\mathfrak{q}_v({\rm G})$ & the integrability exponent of automorphic representations, see
(\ref{eq:q_v}),\\
$\sigma_S$ & the exponent of a subset $S$ of $V_K$, see (\ref{eq:sigma_S}),\\
$\mathfrak{q}_S({\rm G})$ & see (\ref{eq:q_s}),\\
$\Upsilon$ & $G_{V_K}/{\rm G}(K)$,\\
$\mu$ & the probability Haar measure on $\Upsilon$,\\
$\mathcal{X}_{aut}(G_{V_K})$ & the set of automorphic characters of $G_{V_K}$,\\
$\mathcal{X}_{aut}(G_{V_K})^U$ & the set of automorphic characters invariant under $U$,\\
$G^U$ & the kernel of $\mathcal{X}_{aut}(G_{V_K})^U$ in $G_{V_K}$,\\
$\mathcal{O}_S(\epsilon)$ & the neighborhood of identity in $G_S$, see (\ref{eq:o_s}),\\
$B_h$ & see (\ref{eq:b__h}).
\end{tabular}

}

\end{document}